\documentclass[preprint,12pt]{elsarticle}

\usepackage{amssymb}
 \usepackage{amsthm}
\usepackage[latin1]{inputenc}
\usepackage{amsmath} 
\usepackage{amsfonts}
\usepackage{stmaryrd}
\usepackage{latexsym} 
\usepackage{graphicx}
\usepackage{subfigure}
\usepackage{color}
\usepackage{hyperref}
\usepackage{verbatim}
\usepackage[all]{xy}
\usepackage{graphics}
\usepackage{pdfsync}
\usepackage{tikz}
\usepackage{url}
\usepackage{enumerate}

\theoremstyle{definition}
\newtheorem{theorem}{Theorem}[section]
\newtheorem*{theorem*}{Theorem}
\newtheorem{proposition}[theorem]{Proposition}

\newtheorem{lemma}[theorem]{Lemma}
\newtheorem{corollary}[theorem]{Corollary}

\newtheorem{definition}[theorem]{Definition}
\newtheorem{example}[theorem]{Example}

\theoremstyle{remark}
\newtheorem{remark}[theorem]{Remark}

\newcommand{\SvW}[1]{{#1}}
\newcommand{\rev}[1]{\textcolor{black}{#1}}
\newcommand{\finrev}[1]{\textcolor{black}{#1}}

\newcommand{\type}{\operatorname{type}}
\newcommand{\maxdeg}{\operatorname{maxdeg}}
\newcommand{\width}{\operatorname{maxlen}}

\newcommand{\Sym}{\ensuremath{\operatorname{Sym}}}

\journal{Advances in Applied Mathematics}

\begin{document}

\begin{frontmatter}



\title{Extended chromatic symmetric functions and equality of ribbon Schur functions}


\author[1]{Farid Aliniaeifard}
\ead{farid@math.ubc.ca}
\author[1]{Victor Wang}
\ead{victor.yz.wang@alumni.ubc.ca}
\author[1]{Stephanie van Willigenburg\corref{cor1}}
\cortext[cor1]{Corresponding author}
\ead{steph@math.ubc.ca}

\address[1]{
Department of Mathematics,
University of British Columbia,
Vancouver BC V6T 1Z2,
Canada}

\begin{abstract}
We prove a general inclusion-exclusion relation for the extended chromatic symmetric function of a weighted graph, which specializes to (extended) $k$-deletion, and we give two methods to obtain numerous new bases from weighted graphs for the algebra of symmetric functions. 

Moreover, we classify when two weighted paths have equal extended chromatic symmetric functions by proving this is equivalent to the classification of equal ribbon Schur functions. This latter classification is dependent on the operation composition of compositions, which we generalize to composition of graphs. We then apply our generalization to obtain infinitely many families of weighted graphs whose members have equal extended chromatic symmetric functions. 
\end{abstract}




\end{frontmatter}


\tableofcontents

\section{Introduction}\label{sec:intro}  
The chromatic polynomial of a graph $G$, denoted by $\chi_G$, was introduced in 1912 by Birkhoff as a tool for solving the 4-colour problem \cite{Birk}. Stanley generalized this in 1995 to the chromatic symmetric function, denoted by $X_G$ \cite{Stan95}, and since then it has been an active area of research, gaining particular prominence recently due to two avenues. One avenue of research is to answer the question of whether the chromatic symmetric function distinguishes nonisomorphic trees \cite[p. 170]{Stan95}. This has been confirmed for all trees with $\leq 29$ vertices by Heil and Ji \cite{HeilJi}, and for various infinite families such as spiders by Martin, Morin and Wagner \cite{MMW}, \rev{ proper caterpillars by Aliste-Prieto and Zamora \cite{Jose2}, and all caterpillars by Loebl and Sereni \cite{LS}.} The proof of \rev{Aliste-Prieto and Zamora's} result hinges on the classification of equal ribbon Schur functions by Billera, Thomas and van Willigenburg \cite{BTvW}, which also intriguingly arises in the proof by Borodin, Diaconis and Fulman that a stationary one-dependent process is invariant under time reversal \cite{BDF}. However, the main avenue of research is to prove the Stanley-Stembridge conjecture \cite[Conjecture 5.5]{StanStem}, which in terms of chromatic symmetric functions was given by Stanley \cite[Conjecture 5.1]{Stan95}: if a poset is $(\mathbf{3}+\mathbf{1})$-free then the chromatic symmetric function of its incomparability graph is a nonnegative linear combination of elementary symmetric functions.
Although the conjecture remains open, many cases have been proved, for example in \cite{ChoHuh, Dladders, GebSag, MM, HuhNamYoo}, and as a direction towards proving the conjecture\rev{,} a variety of generalizations of the chromatic symmetric function have been introduced and studied, such as to quasisymmetric functions by Shareshian and Wachs \cite{SW} and to symmetric functions in noncommuting variables by Gebhard and Sagan \cite{GebSag}.

Recently, a further generalization was introduced by Crew and Spirkl \cite{CS}, the extended chromatic symmetric function of a vertex-weighted graph, $(G,w)$, denoted by $X_{(G,w)}$. This function reduces to Stanley's chromatic symmetric function when the weight of every vertex is 1. Their  motivation was to provide a symmetric function that satisfied a deletion-contraction rule analogous to the famed deletion-contraction rule satisfied by the chromatic polynomial, but \emph{not} the chromatic symmetric function. Further to providing such a rule, they also generalized numerous results from $X_G$ to $X_{(G,w)}$. In our paper we build on their work by generalizing further results from $X_G$ to $X_{(G,w)}$, and investigate when two  {vertex-weighted} graphs have the same extended chromatic symmetric function. Remarkably, the equality of ribbon Schur functions again plays a crucial role. More precisely, our paper is structured as follows.

We cover the necessary background in Section~\ref{sec:prelim}. Then in Section~\ref{sec:incexc} we introduce  expansions of a weighted graph in Definition~\ref{def:expansion} and use them to prove an inclusion-exclusion relation for $X_{(G,w)}$ in Theorem~\ref{the:incexc}, which we  {also specialize to $\chi_G$} in Corollary~\ref{cor:aRb}.  {Our result generalizes the $k$-deletion relations for $X_{(G,w)}$ \cite{CS} and $X_G$ \cite{DvW, OS}.} In Section~\ref{sec:compcomp} we work towards classifying when two weighted paths have equal extended chromatic symmetric functions. We give our classification in Theorem~\ref{the:wpclass} that shows it is \rev{equivalent} to the classification of equal ribbon Schur functions \cite{BTvW}. The proof of our classification is reliant on a map $U$, given in \eqref{eq:Umap}, which maps the ribbon Schur function indexed by a composition $\alpha$ to the extended chromatic symmetric function of a path with naturally ordered vertex weights \rev{given by} $\alpha$. In Section~\ref{sec:bases} we give a formula for   extended chromatic symmetric functions in terms of power sum symmetric functions and a M\"{o}bius function in Proposition~\ref{prop:mup}, which naturally generalizes \cite[Theorem 2.6]{Stan95}. We then use this to generate multiplicative bases for the algebra of symmetric functions, $\Sym$, from extended chromatic symmetric functions  in Theorem~\ref{the:wbase}. In Theorem~\ref{the:wbas} we give a second way to generate bases for $\Sym$ from extended chromatic symmetric functions. In Section~\ref{sec:neatbases} we introduce chromatic reciprocity in Theorem~\ref{the:rec} and use it to prove simple change of basis formulae from the power sum symmetric functions to the bases of $\Sym$ generated by the classic chromatic symmetric functions of paths in Proposition~\ref{prop:ppwr} and stars in Proposition~\ref{prop:spwr}. Finally, in Section~\ref{sec:compgra} we define an operation to compose graphs, which allows us to generate nonisomorphic weighted graphs with equal extended chromatic symmetric functions in Theorem~\ref{the:eq}.

\section{Preliminaries}\label{sec:prelim}
A \textit{composition} $\alpha = (\alpha_1,\dots,\alpha_{\ell(\alpha)})$ is a (possibly empty) finite ordered list of positive integers, where $\ell(\alpha)$ is the \textit{length} of $\alpha$. We call the integers the \textit{parts} of the composition. When $\alpha_{j+1}=\cdots=\alpha_{j+m}=i$, we often abbreviate this sublist to $i^m$. The \textit{size}  of $\alpha$ is defined to be $|\alpha|=\alpha_1+\cdots+\alpha_{\ell(\alpha)}$. If $|\alpha|=n$, we say that $\alpha$ is a composition of $n$ and write $\alpha\vDash n$.

Let $[n]=\{1,\dots,n\}$. If $\alpha= (\alpha_1,\dots,\alpha_{\ell(\alpha)})\vDash n$, then we define $\text{set}(\alpha)=\{\alpha_1, \rev{\alpha_1 +\alpha _2,} \dots,\alpha_1+\cdots+\alpha_{\ell(\alpha)-1}\}\subseteq [n-1]$. This induces a natural one-to-one correspondence between the compositions of $n$ and the subsets of $[n-1]$. Define the \textit{complement} of $\alpha$, denoted by $\alpha^c$, to be the unique composition of $n$ satisfying $\text{set}(\alpha^c)=[n-1]-\text{set}(\alpha)$. The \textit{reversal} of $\alpha$ is the composition $\alpha^r=(\alpha_{\ell(\alpha)},\dots,\alpha_1)$.

A \textit{partition} $\lambda = (\lambda_1,\dots,\lambda_{\ell(\lambda)})$ is a composition with entries satisfying $\lambda_1\ge\cdots\ge\lambda_{\ell(\lambda)}$. If $|\lambda|=n$, then we say that $\lambda$ is a partition of $n$ and write $\lambda\vdash n$. The \textit{underlying partition} of a composition $\alpha$, denoted by $\widetilde\alpha$, is the composition obtained by sorting the parts of $\alpha$ into weakly decreasing order. 

Given two compositions $\alpha=(\alpha_1,\dots,\alpha_{\ell(\alpha)})$ and $\beta=(\beta_1,\dots,\beta_{\ell(\beta)})$, the \textit{concatenation} of $\alpha$ and $\beta$ is $\alpha\cdot\beta = (\alpha_1,\dots,\alpha_{\ell(\alpha)},\beta_1,\dots,\beta_{\ell(\beta)})$, while their \textit{near concatenation} is $\alpha\odot\beta = (\alpha_1,\dots,\alpha_{\ell(\alpha)}+\beta_1,\dots,\beta_{\ell(\beta)})$. If $\ell(\alpha)=\ell(\beta)=\ell$ and $\alpha_1\ge\beta_1,\dots,\alpha_\ell\ge\beta_\ell$, then we say that $\alpha$ \textit{contains} $\beta$, and write $\alpha \supseteq \beta$. Finally, $\alpha$ is a \textit{coarsening} of $\beta$  (or equivalently $\beta$ is a \textit{refinement} of $\alpha$), denoted by $\alpha \succcurlyeq\beta$, if we can obtain the parts of $\alpha$ in order by adding together adjacent parts of $\beta$ in order. If $\alpha,\beta$ are compositions of the same size, $\alpha\succcurlyeq\beta$ (or equivalently, $\alpha^c\preccurlyeq \beta^c$) if and only if $\text{set}(\beta)\supseteq\text{set}(\alpha)$.

We next turn our attention to $\Sym$, the \textit{algebra of symmetric functions}. We can define $\Sym$ as a subalgebra of $\mathbb{Q}[[x_1,x_2,\dots]]$ as follows. The \textit{$i$th elementary symmetric function $e_i$} for $i\ge 1$ is defined to be
\begin{equation*}
    e_i=\sum_{j_1<\cdots<j_i}x_{j_1}\cdots x_{j_i}.
\end{equation*}
Given a partition $\lambda=(\lambda_1,\dots,\lambda_{\ell(\lambda)})$, we define the \textit{elementary symmetric function $e_\lambda$} to be
\begin{equation*}
    e_\lambda=\prod_{i=1}^{\ell(\lambda)} e_{\lambda_i}
\end{equation*}
while taking the convention that the empty product evaluates to $1$. $\Sym$ can be defined as the graded algebra
\begin{equation*}
    \Sym=\Sym^0\oplus \Sym^1\oplus\cdots
\end{equation*}
where for each $n\ge 0$, $\Sym^n$ is spanned by the basis $\{e_\lambda\}_{\lambda\vdash n}$. Thus the family $\{e_\lambda\}_{\lambda \vdash n\ge 0}$ forms a basis for $\Sym$. 

A basis $\{b_\lambda\}_{\lambda\vdash n\ge 0}$ of $\Sym$ indexed by   partitions $\lambda$ is \textit{multiplicative} if for each $n\ge 0$ we have $\Sym^n=\text{span}\{b_\lambda\}_{\lambda\vdash n}$, and for every partition $\lambda=(\lambda_1,\dots,\lambda_{\ell(\lambda)})$,
\begin{equation*}
    b_\lambda = \prod_{i=1}^{\ell(\lambda)} b_{\lambda_i}.
\end{equation*}

As such, we can describe a multiplicative basis $\{b_\lambda\}_{\lambda\vdash n\ge 0}$ by simply giving the formulae for each $b_i$ for $i\ge 1$. Aside from the basis of elementary symmetric functions, there are two other important multiplicative bases for $\Sym$.

The first is the basis of \textit{complete homogeneous symmetric functions} $\{h_\lambda\}_{\lambda\vdash n\ge 0}$, which is the multiplicative basis given by
\begin{equation*}
   h_i=\sum_{j_1\le \cdots\le j_i}x_{j_1}\cdots x_{j_i}.
\end{equation*}

The other is the basis of \textit{power sum symmetric functions} $\{p_\lambda\}_{\lambda\vdash n\ge 0}$, which is the multiplicative basis given by
\begin{equation*}
     p_i=\sum_{j}x_j^i.
\end{equation*}

Another class of symmetric functions that we will be interested in are the \textit{ribbon Schur functions}, indexed by compositions, which can be defined in terms of the complete homogeneous symmetric functions via
\begin{equation*} 
    r_\alpha = \sum_{\beta \succcurlyeq \alpha} (-1)^{\ell(\alpha)-\ell(\beta)} h_{\widetilde \beta}.
\end{equation*}

$\Sym$ became an object of study in graph theory when Stanley introduced a symmetric function generalization of the chromatic polynomial of a graph.

\rev{Let $G$ be a graph  with finite vertex set $V(G)$ and finite multiset of edges $E(G)$. For $u,v\in V(G)$, we  write $uv$ to mean an edge connecting $u$ and $v$. We henceforth assume that all our graphs are  finite. A \textit{proper colouring} of $G$ is an assignment of colours to the vertices of $G$ such that no two vertices connected by an edge are given the same colour. Equivalently, a proper colouring is a map $\kappa: V(G)\to\mathbb{Z}^+$ such that $\kappa(u)\neq \kappa(v)$ when $u,v\in V(G)$ and $uv\in E(G)$. For $k\ge 0$, the function $\chi_G(k)$ denotes the number of proper colourings of $G$ using $k$ colours. It is perhaps a surprising result that $\chi_G(k)$ is polynomial in $k$; as such, $\chi_G$ is known as the \textit{chromatic polynomial}. }

 We permit our graphs to have \textit{loops} (edges connecting some vertex to itself) and possibly \textit{multiple edges} (two or more edges incident to the same pair of vertices); as we shall soon see, allowing loops and multiple edges in our graphs will still lead to interesting results.  A graph with no loops and no multiple edges is \textit{simple}. We will sometimes require our graphs to be \textit{labelled}, namely that the vertices of our graphs are assigned a canonical ordering $v_1,\dots,v_{|V(G)|}$. 

We will require familiarity with a few families of graphs, which we describe here. The \textit{path} $P_n$, $n\geq 1$, is the graph on $n$ vertices $v_1,\dots,v_n$ with edge set $\{v_iv_{i+1}\mid i\in[n-1]\}$, and the \textit{star} $S_n$, $n\geq 1$, is the graph on $n$ vertices $v_1,\dots,v_n$ with edge set $\{v_iv_n \mid i\in[n-1]\}$. When we refer to $P_n$ as a labelled graph, we will always adopt this labelling, which orders the vertices of $P_n$ as they appear along the path. The \emph{null graph} ${N_n}$, $n\geq 1$, is the graph on $n$ vertices with no edges.

Given two graphs $G$ and $H$, we write $G\cup H$ to mean their disjoint union. When $G$ and $H$ are labelled graphs with vertices ordered $a_1,\dots,a_n$ and $b_1,\dots,b_m$, respectively, we define the labelled graph $G\mid H$ to be the graph of their disjoint union on labelled vertices $v_1,\dots,v_{n+m}$ such that $v_i=a_i$ for $1\le i \le n$ and $v_{i}=b_{i-n}$ for $n+1\le i \le n+m$.

In 1995, Stanley generalized the chromatic polynomial of a graph $G$ by defining the chromatic symmetric function of $G$ as follows.
\begin{definition}\label{def:XG}
\cite[Definition 2.1]{Stan95}
Let $G$ be a graph with vertex set $\{v_1,\dots,v_n\}$. Then the \textit{chromatic symmetric function} of $G$ is defined to be
\begin{equation*}
    X_G = \sum_\kappa x_{\kappa(v_1)}\cdots x_{\kappa(v_n)}
\end{equation*}
where the sum is over all proper colourings $\kappa$ of $G$.
\end{definition}

The chromatic symmetric function $X_G$ of a graph $G$ specializes to $\chi_G(k)$ when evaluated at $x_i=1$ for $i\le k$ and $x_i=0$ for $i>k$.

In 2020, Crew and Spirkl introduced a natural extension of the chromatic symmetric function to a \textit{weighted graph} $(G,w)$ where $w:V(G)\to\mathbb{Z}^+$ describes the \textit{weight} of each vertex of $G$.

\begin{definition}\label{def:EXG}
\cite[Equation 1]{CS}
Let $(G,w)$ be a weighted graph with vertex set $\{v_1,\dots,v_N\}$ and weight function $w:V(G)\to\mathbb{Z}^+$. Then the \textit{extended chromatic symmetric function} of $(G,w)$ is defined to be
\begin{equation*}
    X_{(G,w)} = \sum_\kappa x^{w(v_1)}_{\kappa(v_1)}\cdots x_{\kappa(v_N)}^{w(v_N)}
\end{equation*}
where the sum is over all proper colourings $\kappa$ of $G$.
\end{definition}

We will investigate properties of the extended chromatic symmetric function while sometimes choosing to employ an alternative notation. If $G$ is a graph with vertex set $\{v_1,\dots,v_N\}$ and $\alpha$ is a composition of size $n$ and length $N$, we allow ourselves to write $(G,\alpha)$ to denote the weighted graph $(G,w)$ with weight function $w(v_i)=\alpha_i$. This of course implicitly assumes that $G$ is a labelled graph. When needed, we will describe this order explicitly, although we may choose to omit such a description when all possible labellings of the vertices, when combined with the weight composition $\alpha \vDash n$, produce the same weighted graph up to isomorphism. For example, if $\alpha = (1^n)$, then  $X_{(G,(1^n))}=X_G$, regardless of how we label its vertices.

\SvW{With this notation, we immediately have the following by definition.}

\SvW{\begin{proposition}\label{prop:wprod} Let $(G, \alpha)$ and $(H, \beta)$ be weighted graphs. Then
$$X_{(G \mid H, \alpha \cdot \beta)} = X_{(G, \alpha)} X_{(H, \beta)}.$$
\end{proposition}}

When drawing weighted graphs, we will write inside each node the weight of the vertex. When we want to emphasize the ordering on the vertices, we will choose to draw the vertices in order from left to right. We will need to discuss the relations between the graphs we draw; as such, we may enclose a graph drawing in square brackets as a shorthand notation for its extended chromatic symmetric function.

One of the main motivations of Crew and Spirkl for studying vertex-weighted graphs was to obtain a \textit{deletion-contraction rule}, which relates the extended chromatic symmetric function of a weighted graph to those of the weighted graphs obtained from deleting and contracting a fixed edge. 

To \textit{delete} an edge $\epsilon$ of a graph $G$ means to consider the graph $G-\epsilon=(V(G),E(G)-\{\epsilon\})$ resulting from removing $\epsilon$ from the edge multiset of $G$. If $S$ is a multiset of edges contained in $E(G)$, we  similarly use the notation $G-S$ to mean the graph $(V(G),E(G)-S)$. For any multiset of edges $S$ on the vertices of $G$, we let $G+S$ denote the graph $(V(G), E(G)+S)$, and we  write $G+\{\epsilon\}$ as $G+\epsilon$. 

To \textit{contract} an edge $\epsilon$ of $G$, we first delete it from $G$, and then construct the graph $G/\epsilon$ by formally identifying the endpoints of $\epsilon$ as the same vertex. When vertex weights are relevant, we will take the weight of the resulting vertex to be the sum of the weights of its constituents unless otherwise specified. This will be consistent with Proposition~\ref{prop:dc} below. Note that when $G$ is a labelled graph and $\epsilon$ connects two successively ordered vertices of $G$ or is a loop, there is a natural ordering on the vertices of $G/\epsilon$ inherited from $G$. In general, if we contract several edges of a labelled \finrev{graph} $G$ such that each resulting vertex is a combination of consecutively labelled vertices of $G$, we will assume that the new graph inherits the natural ordering on its vertices from $G$.

A deletion-contraction rule exists for the chromatic polynomial, but not for Stanley's original unweighted chromatic symmetric function. We reproduce the statement of the rule in the weighted case here in our alternative notation. Because of our composition notation for a weighted graph, our statement requires an additional condition on the ordering of the vertices, but this ultimately expresses the deletion-contraction rule in its full generality, as we can always relabel the vertices of $G$ to satisfy the required conditions.

\begin{proposition}[Deletion-contraction]
\cite[Lemma 2]{CS}
\label{prop:dc}
Let $(G,\alpha)$ be a weighted graph, where $\alpha=(\alpha_1,\dots,\alpha_{\ell(\alpha)})$ is a composition specifying the weights of the vertices $v_1,\dots,v_{\ell(\alpha)}$ of $G$. Let $\epsilon$ be either an edge connecting consecutively labelled vertices $v_i,v_{i+1}$ of $G$ or a loop. Write $\alpha/\epsilon=(\alpha_1,\dots,\alpha_i)\odot(\alpha_{i+1},\dots,\alpha_{\ell(\alpha)})$ in the first case, and $\alpha/\epsilon=\alpha$ if $\epsilon$ is a loop. Then 
\begin{equation*}
    X_{(G,\alpha)}=X_{(G-\epsilon,\alpha)}-X_{(G/\epsilon,\alpha/\epsilon)}.
\end{equation*}
\end{proposition}

\begin{example}\label{ex:dc}
Take $G$ to be the cycle on $3$ vertices $v_1,v_2,v_3$, and $\alpha=(3,2,1)$, which assigns the weights $w(v_1)=3$, $w(v_2)=2$, $w(v_3)=1$. Take $\epsilon$ to be the edge connecting $v_2$ and $v_3$. Removing $\epsilon$ from the edge set of $G$ gives us $G-\epsilon$. 

To obtain $G/\epsilon$, we take $G-\epsilon$ and identify the endpoints $v_2,v_3$ of $\epsilon$ as a single vertex $v^*$. The edge connecting $v_1$ and $v_2$ becomes an edge connecting $v_1$ and $v^*$. Similarly, the edge between $v_1$ and $v_3$ becomes another edge connecting $v_1$ and $v^*$. Since $v_2,v_3$ have consecutive labels, $G/\epsilon$ inherits a vertex ordering from $G$, with $v_1$ ordered before $v^*$. The weighting on $G/\epsilon$ is given by $\alpha/\epsilon=(3,2)\odot(1)=(3,3)$. 

In our pictorial shorthand, the deletion-contraction rule then gives us the following.
\begin{center}
\begin{tikzpicture}
\node[shape=circle,draw=black] (A) at (0,0) {3};
\node[shape=circle,draw=black] (B) at (1,0) {2};
\node[shape=circle,draw=black] (C) at (2,0) {1};
\path [-] (A) edge (B);
\path [-] (C) edge (B);
\draw (A)to [out=35, in=145] node[above] {} (C);
\node [] at (-.5,0) {\scalebox{2}{[}};
\node [] at (2.5,0) {\scalebox{2}{]}};
\node [] at (3,0) {\scalebox{1}{$=$}};
\node[shape=circle,draw=black] (A') at (4,0) {3};
\node[shape=circle,draw=black] (B') at (5,0) {2};
\node[shape=circle,draw=black] (C') at (6,0) {1};
\path [-] (A') edge (B');
\draw (A')to [out=35, in=145] node[above] {} (C');
\node [] at (3.5,0) {\scalebox{2}{[}};
\node [] at (6.5,0) {\scalebox{2}{]}};
\node[shape=circle,draw=black] ('A) at (8,0) {3};
\node[shape=circle,draw=black] ('B) at (9,0) {3};
\path [-] ('A) edge ('B);
\path [-] ('A) edge[bend left] ('B);
\node [] at (7.5,0) {\scalebox{2}{[}};
\node [] at (9.5,0) {\scalebox{2}{]}};
\node [] at (7,0) {\scalebox{1}{$-$}};
\end{tikzpicture}
\end{center}
\end{example}

We are now ready to discuss the results of the paper.

\section{An inclusion-exclusion relation}\label{sec:incexc}
Our first theorem is a useful expansion that relates the extended chromatic symmetric function of a weighted graph to those of certain other weighted graphs. Applications of this theorem will allow us to prove the results in later sections of our paper. Before stating the relation, we first give some necessary definitions.

\begin{definition}\label{def:aRb}
Given weighted graphs $(G,\alpha)$ and $(H,\beta)$ satisfying $\beta\preccurlyeq\alpha$, for a vertex $v$ of $G$, we write $R(v)$ to denote the set of consecutively labelled vertices in $H$ whose weights are summed to obtain the weight of $v$ when describing $\alpha$ as a coarsening of $\beta$. 

Note that this induces an equivalence relation on the vertices of $H$, whose equivalence classes correspond to vertices of $G$. We write this as $aRb$ for vertices $a,b$ of $H$ if and only if there exists vertex $v$ of $G$ such that $a,b$ are both in $R(v)$.
\end{definition}

\begin{example}\label{ex:aRb}
Let $G$ be a graph on 3 vertices labelled $v_1,v_2,v_3$ with weights given by $\alpha=(5,3,9)$ and $H$ be a graph on 5 vertices labelled $a_1,a_2,a_3,a_4,a_5$ with weights given by $\beta=(1,4,3,7,2)$. Then we can write $\alpha$ as a coarsening of $\beta$ via $\alpha=(\beta_1+\beta_2,\beta_3,\beta_4+\beta_5)$.

Hence, $R(v_1)=\{a_1,a_2\}$, $R(v_2)=\{a_3\}$, and $R(v_3)=\{a_4,a_5\}$. The equivalence relation $R$ on the vertices of $H$ is given by the reflexive symmetric transitive closure of the relations $a_1Ra_2$ and $a_4Ra_5$.
\end{example}

\begin{definition}\label{def:expansion}
Let $(G,\alpha)$ be a weighted graph. Then we say $(H,\beta)$ is an \textit{expansion} of $(G,\alpha)$ if
\begin{enumerate}[1.]
    \item the composition $\beta$ is a refinement of $\alpha$, and
    \item for all pairs of vertices $u,v$ (not necessarily distinct) of $G$, there is an edge $uv$ in $E(G)$ if and only if there exists an edge $ab$ in $E(H)$ with $a\in R(u)$ and $b\in R(v)$.
\end{enumerate}
\end{definition}

\begin{example}\label{ex:exp}
Drawn below, the weighted graph $(H,(3,2,3))$ on vertices $a_1,a_2,a_3$, labelled from left to right, is an expansion of the weighted graph $(G,(3,5))$ on vertices $v_1,v_2$, labelled from left to right.

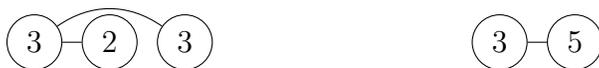
\begin{center}
\begin{tikzpicture}
\node[shape=circle,draw=black] (A) at (0,0) {3};
\node[shape=circle,draw=black] (B) at (1,0) {2};
\node[shape=circle,draw=black] (C) at (2,0) {3};
\path [-] (A) edge (B);
\draw (A)to [out=35, in=145] node[above] {} (C);
\end{tikzpicture}
\quad \quad \quad \quad \quad \quad \quad \quad
\begin{tikzpicture}
\node[shape=circle,draw=black] (A) at (0,0) {3};
\node[shape=circle,draw=black] (B) at (1,0) {5};
\path [-] (A) edge (B);
\end{tikzpicture}
\end{center}

First, we see that $(3,2,3)$ is a refinement of $(3,5)=(3,2+3)$, so $R(v_1)=\{a_1\}$ and $R(v_2)=\{a_2,a_3\}$. Since $G$ has no loop on $v_1$, there cannot be a loop on $a_1$. Because $G$ has an edge connecting $v_1$ and $v_2$, we must have at least one edge of $H$ between $a_1$ (the only element of $R(v_1)$) and one of $a_2,a_3$ (the elements of $R(v_2)$); this condition is satisfied by both edges of $H$. Finally, $G$ has no loop on $v_2$, so $H$ cannot have a loop on either of $a_2,a_3$, nor an edge connecting $a_2$ and $a_3$.

Note we could have omitted either of the two edges of $(H,(3,2,3))$, and still obtained an expansion of $(G,(3,5))$, but not both.
\end{example}

With these definitions in mind, we state our first theorem.

\begin{theorem}\label{the:incexc}
Let $(G,\alpha)$ be a weighted graph with expansion $(H,\beta)$. Let $E'$ be a multiset of edges on the vertices of $H$ such that for each pair of vertices $a,b$ of $H$, we have $a$ and $b$ in the same connected component of $(V(H),E')$ if and only if $aRb$. Then 
\begin{equation*}
    X_{(G,\alpha)}=\sum_{S\subseteq E'}(-1)^{|S|} X_{(H+S,\beta)}.
\end{equation*}
\end{theorem}
\begin{proof}
Because ($H,\beta$) is an expansion of $(G,\alpha)$, there is a natural bijection between the proper colourings of $(G,\alpha)$ and the proper colourings of $(H,\beta)$ assigning the same colour to vertices $a,b$ of $H$ whenever $aRb$: given a proper colouring $\kappa$ on $(G,\alpha)$, simply assign for each vertex $v$ of $G$ the colour $\kappa(v)$ to all vertices in $R(v)$. In particular, a proper colouring of $(G,\alpha)$ contributes the same monomial to $X_{(G,\alpha)}$ as its image under the bijection does to $X_{(H,\beta)}$.

Because the connected components of $(V(H),E')$ correspond to the equivalence classes induced by $R$, the proper colourings of $H$ that assign a single colour to each equivalence class of $R$ are exactly the proper colourings of $H$ that give the endpoints of $\epsilon$ the same colour for each edge $\epsilon$ in $E'$. These colourings can be thought of as the proper colourings of $H$ excluding those that assign different colours to the endpoints of $\epsilon$ for any edge $\epsilon$ in $E'$, or equivalently, the proper colourings of $H$ that are not proper colourings of $H+\epsilon$ for any $\epsilon$ in $E'$.

Given a nonempty collection of edges $\emptyset\subsetneq S \subseteq E'$, the intersection of the proper colourings of $H+\epsilon$ over all edges $\epsilon$ in $S$ gives exactly all the proper colourings of $H$ that also satisfy that no edge of $S$ has endpoints assigned the same colour\,---\,namely, all the proper colourings of $H+S$.

Applying the principle of inclusion-exclusion, we obtain
\begin{equation*}
    X_{(G,\alpha)}=X_{(H,\beta)}-\sum_{\emptyset\subsetneq S\subseteq E'}(-1)^{|S|-1} X_{(H+S,\beta)}=\sum_{S\subseteq E'}(-1)^{|S|} X_{(H+S,\beta)}.
\end{equation*}
\end{proof}

\begin{example}\label{ex:exp2}
We saw in Example~\ref{ex:exp} a weighted graph $(G,(3,5))$ on vertices $v_1,v_2$ with expansion $(H,(3,2,3))$ on vertices $a_1,a_2,a_3$. Let us take $E'$ to be a pair of edges both connecting vertices $a_2$ and $a_3$ of $H$.

The conditions of Theorem~\ref{the:incexc} are satisfied, because the connected components of $(V(H),E')$ partition $V(H)$ into the sets $\{a_1\}$ and $\{a_2,a_3\}$, which are $R(v_1)$ and $R(v_2)$, respectively.

Pictorially, we can write the result of applying Theorem~\ref{the:incexc} as follows.
\begin{center}
\begin{tikzpicture}
\node[shape=circle,draw=black] (B) at (.5,0) {3};
\node[shape=circle,draw=black] (C) at (1.5,0) {5};
\path [-] (C) edge (B);
\node [] at (0,0) {\scalebox{2}{[}};
\node [] at (2,0) {\scalebox{2}{]}};
\node [] at (2.5,0) {\scalebox{1}{$=$}};
\node[shape=circle,draw=black] (A') at (4,.5) {3};
\node[shape=circle,draw=black] (B') at (5,.5) {2};
\node[shape=circle,draw=black] (C') at (6,.5) {3};
\path [-] (A') edge (B');
\draw (A')to [out=35, in=145] node[above] {} (C');
\node [] at (3.5,.5) {\scalebox{2}{[}};
\node [] at (6.5,.5) {\scalebox{2}{]}};
\node[shape=circle,draw=black] ('A) at (8,.5) {3};
\node[shape=circle,draw=black] ('B) at (9,.5) {2};
\node[shape=circle,draw=black] ('C) at (10,.5) {3};
\path [-] ('A) edge ('B);
\path [-] ('B) edge ('C);
\draw ('A)to [out=35, in=145] node[above] {} ('C);
\node [] at (7.5,.5) {\scalebox{2}{[}};
\node [] at (10.5,.5) {\scalebox{2}{]}};
\node [] at (7,.5) {\scalebox{1}{$-$}};
\node [] at (3,-.5) {\scalebox{1}{$-$}};
\node[shape=circle,draw=black] (A') at (4,-.5) {3};
\node[shape=circle,draw=black] (B') at (5,-.5) {2};
\node[shape=circle,draw=black] (C') at (6,-.5) {3};
\path [-] (A') edge (B');
\draw (A')to [out=35, in=145] node[above] {} (C');
\path [-] (B') edge[bend right] (C');
\node [] at (3.5,-.5) {\scalebox{2}{[}};
\node [] at (6.5,-.5) {\scalebox{2}{]}};
\node[shape=circle,draw=black] ('A) at (8,-.5) {3};
\node[shape=circle,draw=black] ('B) at (9,-.5) {2};
\node[shape=circle,draw=black] ('C) at (10,-.5) {3};
\path [-] ('A) edge ('B);
\path [-] ('B) edge ('C);
\draw ('A)to [out=35, in=145] node[above] {} ('C);
\path [-] ('B) edge[bend right] ('C);
\node [] at (7.5,-.5) {\scalebox{2}{[}};
\node [] at (10.5,-.5) {\scalebox{2}{]}};
\node [] at (7,-.5) {\scalebox{1}{$+$}};
\end{tikzpicture}
\end{center}
\end{example}

\begin{corollary}\label{cor:aRb}
Let $(G,\alpha)$ be a weighted graph with expansion $(H,\beta)$. Let $E'$ be a multiset of edges on the vertices of $H$ such that for each pair of vertices $a,b$ of $H$, we have $a$ and $b$ in the same connected component of $(V(H),E')$ if and only if $aRb$. Then 
\begin{equation*}
    \chi_{G}=\sum_{S\subseteq E'}(-1)^{|S|} \chi_{H+S}.
\end{equation*}
\end{corollary}

\begin{proof}
\rev{First note that $X_{(G,w)}$ specializes to $\chi_{G}(k)$, when evaluated at $x_i=1$ for $i\leq k$ and $x_i=0$ otherwise. Hence, our equality} holds on every positive integer $k$ by evaluating the formula in Theorem~\ref{the:incexc} at $x_i=1$ for $i\le k$ and $x_i=0$ otherwise. Because the real polynomials on both sides of the equation agree on infinitely many values, they must be equal.
\end{proof}

Theorem~\ref{the:incexc} also generalizes a related result known as \textit{$k$-deletion}, which we state next, and give a new short and simple proof.

\begin{corollary}[$k$-deletion]\label{cor:kdel} 
\cite[Theorem 6]{CS}
Let $(G,\alpha)$ be a weighted graph containing a cycle $C$ on $k$ vertices, and let $\epsilon$ be a fixed edge of this cycle. Then 
\begin{equation*}
    \sum_{S\subseteq E(C)- \epsilon}(-1)^{|S|}X_{(G- S,\alpha)} = 0.
\end{equation*}
\end{corollary}

\begin{proof}
We can assume without loss of generality (by relabelling the vertices as needed), that $C$ is a cycle on the first $k$ vertices $v_1,\dots,v_k$ of the labelled graph $G$. Let $(H,\beta)$ be the weighted graph obtained by contracting (in any order) the edges $E(C)- \epsilon$ of $G$. Note that the image of $\epsilon$ in the new graph is a loop on the resulting vertex, and so $X_{(H,\beta)}=0$.

Taking $E'=E(C)- \epsilon$, we see that $(G- E',\alpha)$ is an expansion of $(H,\beta)$ satisfying the conditions of Theorem~\ref{the:incexc} with the set $E'$. Applying the theorem, we obtain
\begin{equation*}
    \sum_{S\subseteq E(C)- \epsilon}(-1)^{|S|}X_{(G- S,\alpha)} = (-1)^{|E'|} X_{(H,\beta)}= 0.
\end{equation*}
\end{proof}

\begin{remark}\label{rem:kdel}
In the absence of a deletion-contraction rule for the unweighted chromatic symmetric function, the technique of $k$-deletion was developed and generalized from its original form across several different papers as \rev{a} way to write the chromatic symmetric function of a graph as a combination of the chromatic symmetric functions of other graphs.

In 2014, Orellana and Scott discovered and proved the triple-deletion rule \cite[Theorem 3.1]{OS}, which is the case of $k$-deletion on unweighted graphs for $k=3$, by directly expanding the $2^{3-1}=4$ terms of the summation into the power sum symmetric functions and computing. In a 2018 paper, Dahlberg and van Willigenburg generalized the result of Orellana and Scott on unweighted graphs to arbitrary $k$ \cite[Proposition 5]{DvW} by applying a sign-reversing involution to the terms of the expansion.

When Crew and Spirkl introduced the extended chromatic symmetric function, they were able to prove Corollary~\ref{cor:kdel} via induction \cite[Theorem 6]{CS}, employing repeated applications of the deletion-contraction rule.

Our proof of weighted $k$-deletion is novel in that it is not only simple, but also provides a  combinatorial interpretation as to why the result should hold at all: the summation in question evaluates to $0$ because it describes (up to a sign) all the proper colourings of a certain weighted graph with a loop\,---\,of which there are exactly none.
\end{remark}

\section{Composition of compositions and equality of weighted paths}\label{sec:compcomp}
It is an open problem whether if $G$ and $H$ are two trees with $X_G=X_H$ then $G$ and $H$ are necessarily isomorphic as graphs. In their paper introducing the extended chromatic symmetric function \cite{CS}, Crew and Spirkl \rev{gave an example of} two weighted trees with the same extended chromatic symmetric function that were nonisomorphic as weighted graphs, which was originally noted by Loebl and Sereni in \cite{LS}. The example they \rev{gave} in \cite[Figure 1]{CS} compared two $5$-vertex paths: one with weights $1,2,1,3,2$ in given order, and another with weights $1,3,2,1,2$ in given order.

In our composition notation for weighted graphs, they found that
\begin{equation*}
    X_{(P_5,(1,2,1,3,2))}=X_{(P_5,(1,3,2,1,2))}.
\end{equation*}

It is a curious coincidence, then, that \rev{we have the  equality of ribbon Schur functions}
\begin{equation*}
    r_{(1,2,1,3,2)}=r_{(1,3,2,1,2)}.
\end{equation*}

Another property of the ribbon Schur functions is that they have a simple multiplication rule. For any two nonempty compositions $\alpha,\beta$,
\begin{equation}\label{eq:ribMult}
    r_\alpha r_\beta = r_{\alpha\cdot\beta} + r_{\alpha\odot\beta}.
\end{equation}

The extended chromatic symmetric functions of weighted paths follow the same multiplication rule:
\begin{equation}\label{eq:pathMult}
    X_{(P_{\ell(\alpha)},\alpha)}X_{(P_{\ell(\beta)},\beta)}=X_{(P_{\ell(\alpha)}\mid P_{\ell(\beta)}, \alpha\cdot\beta)}=X_{(P_{\ell(\alpha\cdot\beta)},\alpha\cdot\beta)}+X_{(P_{\ell(\alpha\odot\beta)},\alpha\odot\beta)}.
\end{equation}

The above equality is verified by applying Proposition~\ref{prop:dc}, the deletion-contraction rule: the weighted graph $(P_{\ell(\alpha)}\mid P_{\ell(\beta)},\alpha\cdot\beta)$ can be interpreted as the result of deleting a certain edge of $(P_{\ell(\alpha\cdot\beta)},\alpha\cdot\beta)$, while $(P_{\ell(\alpha\odot\beta)},\alpha\odot\beta)$ would result from contracting that edge.

As we shall soon see, the similarities between the ribbon Schur functions and the extended chromatic symmetric functions of weighted paths are not   superficial. Understanding the connection between them will allow us to, among other things, completely classify when the extended chromatic symmetric functions of two weighted paths are equal.

\begin{definition}\label{def:nifty}
A family $\{G_n\}_{n\ge1}$ of simple connected graphs is \textit{nifty} if each $G_n$ has exactly $n$ vertices. Given a nifty family, we write $G_\lambda$ for a partition $\lambda=(\lambda_1, \dots, \lambda_{\ell(\lambda)})$ to mean the disjoint union of graphs $G_{\lambda_1}\cup\cdots\cup G_{\lambda_{\ell(\lambda)}}$. If we interpret the $G_n$ as labelled graphs, we can also write $G_\alpha$ for a composition $\alpha=(\alpha_1,\dots,\alpha_{\ell(\alpha)})$ to mean the labelled graph $G_{\alpha_1}\mid\dots\mid G_{\alpha_{\ell(\alpha)}}$. As unlabelled graphs, we then always have $G_\alpha = G_{\widetilde\alpha}$.
\end{definition}

\begin{example}\label{ex:nifty}
Nifty families of graphs include the paths $\{P_n\}_{n\ge1}$ and the stars $\{S_n\}_{n\ge1}$.
\end{example}

The following result of Cho and van Willigenburg will be useful to us.

\begin{theorem}\label{the:chr}
\cite[Lemma 3 \& Theorem 5]{ChovW}
Let $\{G_n\}_{n\ge1}$ be a nifty family of graphs. Then $\{X_{G_\lambda}\}_{\lambda\vdash n\ge0}$ is a multiplicative basis for $\Sym$. Moreover, the chromatic symmetric functions $\{X_{G_n}\}_{n\ge1}$ are algebraically independent and freely generate $\Sym$.
\end{theorem}

The above theorem gives us a mechanism to better understand weighted paths:    expand the extended chromatic symmetric function of a weighted path in terms of a basis generated by a nifty family.

\begin{example}\label{ex:chr}
Consider the weighted path on three vertices with weights $2,1,2$ given in order from left to right. Let us rearrange the deletion-contraction rule of Proposition~\ref{prop:dc} to
\begin{equation*}
    X_{(G/\epsilon,\alpha/\epsilon)}=X_{(G-\epsilon,\alpha)}-X_{(G,\alpha)}.
\end{equation*}

We can apply this form of the deletion-contraction rule \SvW{thrice} to obtain the following expansion.
\begin{align*}
    X_{(P_3,(2,1,2))} &= 
    X_{(P_{(3,1)},(2,1,1,1))}-X_{(P_4,(2,1,1,1))}\\
    &=X_{P_{(1,3,1)}}- X_{P_{(4,1)}} -X_{P_{(1,4)}}+X_{P_5}
\end{align*}

We illustrate this expansion below.
\begin{center}
\begin{tikzpicture}
\node[shape=circle,draw=black] (A) at (-.5,0) {2};
\node[shape=circle,draw=black] (B) at (.5,0) {1};
\node[shape=circle,draw=black] (C) at (1.5,0) {2};
\path [-] (C) edge (B);
\path [-] (A) edge (B);
\node [] at (-1,0) {\scalebox{2}{[}};
\node [] at (2,0) {\scalebox{2}{]}};
\node [] at (2.5,0) {\scalebox{1}{$=$}};
\node[shape=circle,draw=black] (A') at (4,.5) {1};
\node[shape=circle,draw=black] (B') at (5,.5) {1};
\node[shape=circle,draw=black] (C') at (6,.5) {1};
\node[shape=circle,draw=black] (D') at (7,.5) {1};
\node[shape=circle,draw=black] (E') at (8,.5) {1};
\path [-] (B') edge (C');
\path [-] (D') edge (C');
\node [] at (3.5,.5) {\scalebox{2}{[}};
\node [] at (8.5,.5) {\scalebox{2}{]}};

\node[shape=circle,draw=black] (A') at (10,.5) {1};
\node[shape=circle,draw=black] (B') at (11,.5) {1};
\node[shape=circle,draw=black] (C') at (12,.5) {1};
\node[shape=circle,draw=black] (D') at (13,.5) {1};
\node[shape=circle,draw=black] (E') at (14,.5) {1};
\path [-] (A') edge (B');
\path [-] (B') edge (C');
\path [-] (D') edge (C');
\node [] at (9.5,.5) {\scalebox{2}{[}};
\node [] at (14.5,.5) {\scalebox{2}{]}};
\node [] at (9,.5) {\scalebox{1}{$-$}};

\node[shape=circle,draw=black] (A') at (4,-.5) {1};
\node[shape=circle,draw=black] (B') at (5,-.5) {1};
\node[shape=circle,draw=black] (C') at (6,-.5) {1};
\node[shape=circle,draw=black] (D') at (7,-.5) {1};
\node[shape=circle,draw=black] (E') at (8,-.5) {1};
\path [-] (B') edge (C');
\path [-] (D') edge (C');
\path [-] (D') edge (E');
\node [] at (3.5,-.5) {\scalebox{2}{[}};
\node [] at (8.5,-.5) {\scalebox{2}{]}};
\node [] at (3,-.5) {\scalebox{1}{$-$}};

\node[shape=circle,draw=black] (A') at (10,-.5) {1};
\node[shape=circle,draw=black] (B') at (11,-.5) {1};
\node[shape=circle,draw=black] (C') at (12,-.5) {1};
\node[shape=circle,draw=black] (D') at (13,-.5) {1};
\node[shape=circle,draw=black] (E') at (14,-.5) {1};
\path [-] (A') edge (B');
\path [-] (B') edge (C');
\path [-] (D') edge (C');
\path [-] (D') edge (E');
\node [] at (9.5,-.5) {\scalebox{2}{[}};
\node [] at (14.5,-.5) {\scalebox{2}{]}};
\node [] at (9,-.5) {\scalebox{1}{$+$}};
\end{tikzpicture}
\end{center}
\end{example}

In the above example, note that the compositions that appear are exactly the coarsenings of $(2,1,2)^c=(1,3,1)$, with the terms alternating in sign depending on the number of parts in the composition. We will prove that this will always be the case for any weighted path, after proving a lemma.

\begin{lemma}\label{lem:comp}
Consider the null graph ${N_n}$ on $n$ vertices $v_1,\dots,v_n$. Let $\alpha$ be a composition of $n$. Then the graph  ${N_n}+\{v_iv_{i+1}\mid i\in \textup{set}(\alpha^c)\}$ is the labelled graph $P_{\alpha}$.
\end{lemma}

\begin{proof}
It suffices to show that the edge set of $P_\alpha$ is $\{v_iv_{i+1}\mid i\in\textup{set}(\alpha^c)\}$. Write $\alpha=(\alpha_1,\dots,\alpha_{\ell(\alpha)})$.

Then the connected components of $P_\alpha$ partition its vertices into the sets $\{v_1,\dots,v_{\alpha_1}\}$, $\{v_{\alpha_1+1},\dots,v_{\alpha_1+\alpha_2}\}$, \dots, $\{v_{\alpha_1+\dots+\alpha_{\ell(\alpha)-1}+1},\dots,v_n\}$. The $i$th connected component of $P_\alpha$ is a copy of the labelled graph $P_{\alpha_i}$ with the same relative ordering of vertex labels. Hence, the edge set of $P_\alpha$ consists of all the edges $v_iv_{i+1}$ for all $i$ in $[n-1]-\{\alpha_1,\alpha_1+\alpha_2,\dots,\alpha_1+\dots+\alpha_{\ell(\alpha)-1}\}$.

That is, $P_\alpha$ has edge set $\{v_iv_{i+1}\mid i\in\textup{set}(\alpha^c)\}$.
\end{proof}

\begin{example}\label{ex:pathset}
Take $n=6$ with $\alpha = (2,3,1)\vDash 6$. Then $\textup{set}(\alpha)=\{2,5\}$, and so $\textup{set}(\alpha^c)=\{1,2,3,4,5\}-\textup{set}(\alpha)=\{1,3,4\}$.

As labelled graphs, $N_6 + \{v_1v_2,v_3v_4,v_4v_5\}=P_2\mid P_3\mid P_1=P_{(2,3,1)}$.
\end{example}

We now turn to prove our expansion for weighted paths.

\begin{proposition}\label{prop:wted}
For any composition $\alpha$,  the extended chromatic symmetric function of the weighted path with weights given, in order, by $\alpha$ is 
\begin{equation}\label{eq:wted}
    X_{(P_{\ell(\alpha)}, \alpha)} = \sum_{\beta \succcurlyeq \alpha^c} (-1)^{\ell(\alpha^c)-\ell(\beta)} X_{P_{\widetilde \beta}}.
\end{equation}\end{proposition}

\begin{proof}
By Definition~\ref{def:expansion},  {one} expansion of $(P_{\ell(\alpha)},\alpha)$ is the graph on $|\alpha|$ vertices $v_1,\dots,v_{|\alpha|}$, each with weight $1$, with edge set $\{v_iv_{i+1}\mid i\in\textup{set}(\alpha)\}$. By Lemma \ref{lem:comp}, this graph is $(P_{\alpha^c},\finrev{(1^{|\alpha|})})$. Taking $E'=\{v_iv_{i+1}\mid i\in [|\alpha|-1]-\textup{set}(\alpha)\}$ satisfies the conditions of Theorem~\ref{the:incexc}, and so we obtain
\begin{align*}
    X_{(P_{\ell(\alpha)},\alpha)} &= \sum_{S\subseteq E'} (-1)^{|S|} X_{P_{\alpha^c}+S}\\
    &=\sum_{\textup{set}(\alpha)\subseteq J \subseteq[|\alpha|-1]} (-1)^{|J-\textup{set}(\alpha)|} X_{N_{|\alpha|}+\{v_iv_{i+1}\mid i\in J\}}.
\end{align*}

We can write each $\textup{set}(\alpha)\subseteq J \subseteq[|\alpha|-1]$ as $J=\textup{set}(\beta^c)$ for some composition $\beta\vDash|\alpha|$ satisfying $\beta^c \preccurlyeq \alpha$, or equivalently, $\beta \succcurlyeq \alpha^c$. Then $|J-\textup{set}(\alpha)|=|\textup{set}(\beta^c)-\textup{set}(\alpha)|=|\textup{set}(\alpha^c)-\textup{set}(\beta)|=(\ell(\alpha^c)-1) - (\ell(\beta)-1)$.

Thus we have that $(-1)^{|J-\textup{set}(\alpha)|}=(-1)^{\ell(\alpha^c)-\ell(\beta)}$. Additionally, by Lemma \ref{lem:comp}, we know $N_{|\alpha|}+\{v_iv_{i+1}\mid i\in J\}=P_{\beta}$.

Hence,
\begin{align*}
    X_{(P_{\ell(\alpha)},\alpha)} &= \sum_{\beta \succcurlyeq \alpha^c}(-1)^{\ell(\alpha^c)-\ell(\beta)} X_{P_{\beta}}\\
    &=\sum_{\beta \succcurlyeq \alpha^c}(-1)^{\ell(\alpha^c)-\ell(\beta)} X_{P_{\widetilde \beta}}
\end{align*}
where the last equality follows because $P_\beta = P_{\widetilde\beta}$ as unlabelled graphs.
\end{proof}

Proposition~\ref{prop:wted} is strongly reminiscent of our earlier definition of ribbon Schur functions:
\begin{equation}\label{eq:rintoh}
    r_\alpha = \sum_{\beta \succcurlyeq \alpha} (-1)^{\ell(\alpha)-\ell(\beta)} h_{\widetilde \beta}.
\end{equation}

There exists a  well-known involutory automorphism of $\Sym$ as a graded algebra, known by $\omega$, which takes $h_\lambda\mapsto e_\lambda$ (and vice-versa) for each partition $\lambda$, as well as $r_\alpha \mapsto r_{\alpha^c}$ for each composition $\alpha$. Applying $\omega$ to both sides of \eqref{eq:rintoh} gives us

\begin{equation}\label{eq:etor}
    r_\alpha = \omega(r_{\alpha^c})= \sum_{\beta \succcurlyeq \alpha^c} (-1)^{\ell(\alpha^c)-\ell(\beta)} e_{\widetilde \beta}.
\end{equation}

Let $U: \Sym \to \Sym$ be the unique linear map taking $e_\lambda \mapsto X_{P_\lambda}$ for each   partition $\lambda$, namely, the following.
\begin{align}\label{eq:Umap}
U:\Sym &\rightarrow \Sym\\
e_\lambda &\mapsto X_{P_\lambda} \nonumber
\end{align}
 Note $\{e_\lambda\}_{\lambda\vdash n\ge 0}$ and $\{X_{P_\lambda}\}_{\lambda\vdash n\ge 0}$ are both multiplicative bases of $\Sym$, so $U$ is well-defined and is an automorphism of $\Sym$ as a graded algebra. By Proposition~\ref{prop:wted} and   \eqref{eq:etor}, we have 
 \begin{equation}\label{eq:rtoXP}U(r_\alpha)=X_{(P_{\ell(\alpha)},\alpha)}\end{equation}for every composition $\alpha$. Hence, results on the ribbon Schur functions apply to the extended chromatic symmetric functions of weighted paths.

\begin{corollary}\label{cor:wpbas}
The family $\{X_{(P_{\ell(\lambda)},\lambda)}\}_{\lambda\vdash n\ge0}$ of extended chromatic symmetric functions of weighted paths indexed by partitions forms a basis for $\Sym$.
\end{corollary}

\begin{proof}
In \cite{BTvW} it is proved that the family of ribbon Schur functions $\{r_\lambda\}_{\lambda \vdash n\ge 0}$ indexed by partitions forms a basis for $\Sym$. Because the linear map $U:\Sym\to\Sym$ takes each $r_\lambda \mapsto  X_{(P_{\ell(\lambda)},\lambda)}$ and is an automorphism of $\Sym$ as a graded algebra, it follows immediately that the family of extended chromatic symmetric functions $\{X_{(P_{\ell(\lambda)},\lambda)}\}_{\lambda\vdash n\ge0}$ of weighted paths indexed by partitions forms a basis for $\Sym$.
\end{proof}

In \cite{BTvW}, Billera, Thomas, and van Willigenburg completely classify when two ribbon Schur functions are equal. We recall a definition and a key result from their paper.

\begin{definition}\label{def:circ}
\cite[Section 3.1]{BTvW}
Given two nonempty compositions $\alpha$ and $\beta$, we define the binary operation $\circ$ by
\begin{equation*}
    \alpha\circ\beta = \beta^{\odot \alpha_1}\cdot\   \cdots\  \cdot\beta^{\odot \alpha_{\ell(\alpha)}}
\end{equation*}
where
\begin{equation*}
    \beta^{\odot i} = \underbrace{\beta \odot\cdots\odot\beta}_i.
\end{equation*}
By \cite[Proposition 3.3]{BTvW}, $\circ$ is associative.
\end{definition}

\begin{example}\label{ex:circ}
Take $\alpha$ and $\beta$ to both be the composition $(1,2)$. Then we have $(1,2)\circ(1,2) = (1,2)^{\odot 1}\cdot (1,2)^{\odot 2} = (1,2,1,3,2)$.
\end{example}

\begin{theorem}\label{the:ribbons}
\cite[Theorem 4.1]{BTvW}
Two nonempty compositions $\alpha$ and $\beta$ satisfy
$r_\alpha = r_\beta$ if and only if for some $\ell$, there exist compositions $\alpha^{(1)},\dots,\alpha^{(\ell)}$ and $\beta^{(1)},\dots,\beta^{(\ell)}$ such that
$$\alpha = \alpha^{(1)} \circ\cdots\circ\alpha^{(\ell)}
\hbox{\rm \quad and \quad}
 \beta = \beta^{(1)} \circ\cdots\circ\beta^{(\ell)}$$
where, for each $i$, either $\beta^{(i)}=\alpha^{(i)}$ or $\beta^{(i)}=(\alpha^{(i)})^r$. We write this equivalence relation as $\alpha\sim\beta$.
\end{theorem}

We are now able to classify when two weighted paths have equal extended chromatic symmetric functions. \SvW{This theorem could also be proved via $\mathcal{L}$-polynomials \cite{Jose2}, however, our proof is direct and manifestly positive.}

\begin{theorem}\label{the:wpclass}
Two nonempty compositions $\alpha$ and $\beta$ satisfy
$$X_{(P_{\ell(\alpha)},\alpha)}=X_{(P_{\ell(\beta)},\beta)}\mbox{ if and only if }\alpha\sim\beta.$$
\end{theorem}

\begin{proof}
Because $U:\Sym\to\Sym$ is an automorphism of $\Sym$ as a graded algebra (in particular, it is injective) and takes each $r_\alpha\mapsto X_{(P_{\ell(\alpha)},\alpha)}$, we have $$U(r_\alpha)=X_{(P_{\ell(\alpha)},\alpha)}=X_{(P_{\ell(\beta)},\beta)}=U(r_\beta)$$if and only if $r_\alpha=r_\beta$. By Theorem~\ref{the:ribbons}, the result follows.
\end{proof}

\begin{example}\label{ex:wpclass}
We saw earlier that $X_{(P_5,(1,2,1,3,2))}=X_{(P_5,(1,3,2,1,2))}$. Note $(1,2,1,3,2)\sim(1,3,2,1,2)$, since $(1,2,1,3,2)=(1,2)\circ(1,2)$, while $(1,3,2,1,2)=(2,1)\circ(1,2)$.
\end{example}

With a little more work, we can deduce the exact number of nonisomorphic weighted paths in each equivalence class of weighted paths with equal extended chromatic symmetric functions. To that end, we present one more definition and one more theorem from \cite{BTvW}.

\begin{definition}\cite[Section 3.2]{BTvW}\label{def:factorization}
If a composition $\alpha$ is written in the form $\alpha^{(1)}\circ\cdots\circ\alpha^{(\ell)}$ then we call this a \textit{factorization} of $\alpha$. A factorization $\alpha=\beta\circ\gamma$ is \textit{trivial} if any of the following hold
\begin{enumerate}[1.]
    \item one of $\beta,\gamma$ is the composition $(1)$,
    \item the compositions $\beta,\gamma$ both have length $1$, or
    \item the compositions $\beta,\gamma$ both have all parts equal to $1$.
\end{enumerate}
Finally, a factorization $\alpha=\alpha^{(1)}\circ\cdots\circ\alpha^{(\ell)}$ is \textit{irreducible} if no $\alpha^{(i)}\circ\alpha^{(i+1)}$ is a trivial factorization, and all factorizations of each $\alpha^{(i)}$ into two compositions are trivial.
\end{definition}

\begin{theorem}\label{the:irr}
\cite[Theorem 3.6]{BTvW}
Every nonempty composition admits a unique irreducible factorization.
\end{theorem}

\begin{example}\label{ex:irr}
The unique irreducible factorization of $(1,2,1,3,2)$ is $(1,2)\circ (1,2)$, since it is not trivial, and the only factorizations of $(1,2)$ into two are the trivial factorizations $(1,2)=(1)\circ(1,2)$ and $(1,2)=(1,2)\circ(1)$.
\end{example}

Thus Billera, Thomas, and van Willigenburg were also able to show in \cite[Theorem 4.1]{BTvW}, that the equivalence class of a nonempty composition $\alpha$ under the equivalence relation $\sim$ contains $2^m$ elements, where $m$ is the number of nonsymmetric (under reversal) terms in the irreducible factorization of $\alpha$.

\begin{corollary}\label{cor:irr} Let $\alpha$ be a nonempty composition with $m$ nonsymmetric terms in its irreducible factorization. If $\alpha \neq \alpha ^r$, then, up to isomorphism, the number of weighted paths $(P, w)$ such that $$X_{(P,w)}= X_{({P_{\ell(\alpha)}},\alpha)}$$is $2^{m-1}$. Otherwise, if $\alpha = \alpha ^r$ then $X_{(P,w)}= X_{(P_{\ell(\alpha)},\alpha)}$ if and only if $(P,w) = (P_{\ell(\alpha)},\alpha)$.
\end{corollary}

\begin{proof} 
\SvW{If $\alpha=\alpha^r$, then the equivalence class of $\alpha$ under $\sim$ contains only itself \cite[Proposition 3.6] {Jose2}, and so $X_{(P,w)}=X_{(P_{\ell(\alpha)},\alpha)}$ if and only if $(P,w)=(P_{\ell(\alpha)},\alpha)$ by Theorem~\ref{the:wpclass}.}

\SvW{Therefore when $\alpha \neq \alpha ^r$, it has no symmetric compositions in its equivalence class. Thus the first part of the statement of the corollary follows immediately from the fact that there are $2^m$ elements in the equivalence class of $\alpha$ under $\sim$, and by noting that each weighted path $(P_{\ell(\alpha)},\alpha)$ is isomorphic to its reversal $(P_{\ell(\alpha^r)},\alpha^r)$.}
\end{proof}

\section{Extended chromatic bases for $\Sym$}\label{sec:bases}
In this section we describe two new ways to generate bases for the algebra $\Sym$ of symmetric functions from the extended chromatic symmetric functions of weighted graphs. One is a generalization of Cho and van Willigenburg's result given in Theorem~\ref{the:chr}, and the other generalizes our weighted path basis found in Corollary~\ref{cor:wpbas}.

The following proposition will be useful.

\begin{proposition}\label{prop:pwr}
\cite[Lemma 3]{CS}
Given a weighted graph $(G,w)$, we can expand it into the power sum symmetric functions via
\begin{equation*}
    X_{(G,w)}=\sum_{S\subseteq E(G)} (-1)^{|S|}p_{\lambda((V(G),S),w)}
\end{equation*}
where for a weighted graph $(G,w)$, the   partition $\lambda(G,w)$ is the partition whose parts are the sums of the vertex weights of each connected component of $(G,w)$.
\end{proposition}

This proposition is the natural generalization of \cite[Theorem 2.5]{Stan95} by Stanley, from $X_G$ to $X_{(G,w)}$. Our next result is the natural generalization of \cite[Theorem 2.6]{Stan95} by Stanley, from $X_G$ to $X_{(G,w)}$ and requires the following definitions.

Given a set partition $\pi = \{ S_1, \dots , S_{\ell(\pi)}\}$ of the vertices of a graph $G$, we say that $\pi$ is \emph{connected} if the restriction of $G$ to each \emph{block} $S_i$ for $1\leq i \leq \ell(\pi)$ is connected.  The \emph{lattice of contractions} of $G$, denoted by $L_G$, is the set of all connected set partitions of $G$, partially ordered by refinement $\leq$. For any $\pi \in L_G$ we have \rev{(for example by  \cite[Equation 1]{Stan95}),}
$$(-1) ^{|V(G)| - \ell(\pi)}\mu (\hat{0}, \pi) >0$$where $\mu$ is the M\"{o}bius function of $L_G$ \rev{and  $\hat{0}$ is the unique minimal element of $L_G$ with each vertex in its own block.} Given a weighted graph $(G,w)$ and a connected set partition of $G$, define $\type(\pi, w)$ to be the partition $\lambda$ of \rev{$\sum _{v\in V(G)} w(v)$ whose parts are the total sums of the weights of each block of $\pi$.}

\begin{proposition}\label{prop:mup}
Given a weighted graph $(G,w)$, we can expand it into the power sum symmetric functions via
\begin{equation*}
X_{(G,w)}=\sum_{\pi \in L_G} \mu (\hat{0}, \pi) p_{\type(\pi, w)}.
\end{equation*} 
\end{proposition}

\begin{proof} For a graph $G$ with $N$ vertices and $\pi \in L_G$ define 
$$X_{(\pi, w)}= \sum_\kappa x^{w(v_1)}_{\kappa(v_1)}\cdots x_{\kappa(v_N)}^{w(v_N)}$$to be the sum over all special colourings $\kappa$ such that for $u,v \in V(G)$
\begin{enumerate}[1.]
\item if $u$ and $v$ are in the same block of $\pi$ then $\kappa(u)=\kappa(v)$
\item if $u$ and $v$ are in different blocks of $\pi$ and there is an edge between $u$ and $v$ then $\kappa(u)\neq\kappa(v)$.
\end{enumerate}
Note that \emph{any} colouring $\kappa$ of $G$ contributes \rev{to a unique} $X_{(\pi, w)}$. We can see this by starting with any colouring $\kappa$ and form each block of its partition $\pi$ by colours, so that all vertices of the same colour are in the same block. Then we refine these blocks   further to respect connected components, so that $\pi$ is a connected set partition of $G$.

Next, by the definition of power sum symmetric functions we have for $\sigma \in L_G$ that
$$p_{\type(\sigma, w)} = \sum _{\pi \in L_G \atop \pi \geq \sigma} X_{(\pi,w)}$$and hence by M\"{o}bius inversion
$$X_{(\sigma,w)}=\sum_{\pi \in L_G \atop \pi \geq \sigma} \mu (\sigma, \pi) p_{\type(\pi, w)}.$$Note that when $\sigma = \hat{0}$ the definition of special colouring coincides with that of proper colouring, so $X_{(\hat{0},w)}= X_{(G,w)}$ and the result follows.
\end{proof}

We now extend the definition of a nifty family to weighted graphs.

\begin{definition}\label{def:wnifty}
A family $\{(G_n,w_n)\}_{n\ge1}$ of simple connected weighted graphs is \textit{nifty} if the sum of the vertex weights of each $(G_n,w_n)$ is exactly $n$. Given a nifty family of weighted graphs, we write $(G_\lambda,w_\lambda)$ for a partition $\lambda=(\lambda_1, \dots, \lambda_{\ell(\lambda)})$ to mean the disjoint union  of weighted graphs $(G_{\lambda_1},w_{\lambda_1})\cup\cdots\cup (G_{\lambda_{\ell(\lambda)}},w_{\lambda_{\ell(\lambda)}})$.
\end{definition}

This suggests the following generalization of Theorem~\ref{the:chr}, which specializes to the unweighted case when each $(G_n,w_n)$ has $n$ vertices of weight $1$. This result was noted independently by Chmutov and Shah, who saw it via Hopf algebraic techniques \cite{Chmutovtalk}. Our proof, however, is combinatorial in nature.

\begin{theorem}\label{the:wbase}
Let $\{(G_n,w_n)\}_{n\ge1}$ be a nifty family of weighted graphs. Then $\{X_{(G_\lambda,w_\lambda)}\}_{\lambda\vdash n\ge0}$ is a multiplicative basis for $\Sym$. Moreover, the extended chromatic symmetric functions $\{X_{(G_n,w_n)}\}_{n\ge1}$ are algebraically independent and freely generate $\Sym$.
\end{theorem}

\begin{proof}
\rev{Let $\lambda=(\lambda_1, \dots, \lambda_{\ell(\lambda)})$ be a partition. Then $$V=\biguplus_{i=1}^{\ell(\lambda)} V(G_{\lambda _i})$$}is the set \rev{of} vertices of $(G _\lambda, w_\lambda)$. By the definition of $(G _\lambda, w_\lambda)$, we know that if $\pi \in L_{G_\lambda}$, then $\type(\pi, w_\lambda)$ equals $\lambda$ or has more parts than $\lambda$. Thus by Propositions~\ref{prop:pwr} and \ref{prop:mup}    it follows that \rev{there exist constants $c_{\lambda\mu}$ such that}
$$X_{(G _\lambda, w_\lambda)}=\sum_{\mu =  \lambda \text{ or } \ell(\mu) > \ell(\lambda)} c_{\lambda\mu} p_\mu$$and, moreover, that $c_{\lambda\lambda}\neq 0$. Hence, \rev{$\{X_{(G_\lambda,w_\lambda)}\}_{\lambda\vdash n\ge0}$} is a multiplicative basis for $\Sym$.

Since for $\lambda=(\lambda_1, \dots, \lambda_{\ell(\lambda)})$, \SvW{by Proposition~\ref{prop:wprod}, we have that
\begin{equation}\label{eq:Glambda}
X_{(G _\lambda, w_\lambda)}=\prod_{i=1}^{\ell(\lambda)} X_{(G_{\lambda_i}, w_{\lambda _i})}
\end{equation}and $\{ X_{(G _\lambda, w_\lambda)}\} _{\lambda \vdash n\geq 0}$} forms a  multiplicative basis for $\Sym$, every element of $\Sym$ is expressible uniquely as a polynomial in the $X_{(G _n, w_n)}$ and hence   the $X_{(G _n, w_n)}$ are algebraically independent and freely generate $\Sym$.
\end{proof}

We can also give a second method of generating bases of $\Sym$ from the extended chromatic symmetric functions of weighted graphs. 

We have seen that the set of extended chromatic symmetric functions of weighted paths indexed by partitions forms a basis of $\Sym$. A natural question to ask is how might we generalize this result? A reasonable hope might be that for all nifty families $\{G_n\}_{n\ge1}$ of unweighted labelled graphs, the set of functions $\{X_{(G_{\ell(\lambda)},\lambda)}\}_{\lambda\vdash n\ge0}$ forms a basis of $\Sym$.

In fact, we prove something more general.

\begin{theorem}\label{the:wbas}
For each   partition $\lambda$, let $H_\lambda$ be an arbitrary \rev{(not necessarily connected)} simple \finrev{labelled} graph on $\ell(\lambda)$ vertices. Then   $\{X_{(H_\lambda, \lambda)}\}_{\lambda\vdash n\ge 0}$ is a basis for $\Sym$.
\end{theorem}

\begin{proof}
It suffices to show for each $n$ that $\{X_{(H_\lambda, \lambda)}\}_{\lambda\vdash n}$ is a basis for $\Sym^n$. To that end, we will show that the change of basis matrix describing the $\{X_{(H_\lambda, \lambda)}\}_{\lambda\vdash n}$ in terms of the power sum symmetric functions of degree $n$ is lower triangular with nonzero entries on the diagonal when indices are given a specific order. We will order the indices by the number of parts in a partition. Between partitions of $n$ of the same length, we  order the indices arbitrarily.

Consider Proposition~\ref{prop:pwr}, which describes how to write $X_{(H_\lambda,\lambda)}$ for some $\lambda \vdash n$ in the basis of power sum symmetric functions. When $S=\emptyset$, the term $p_\lambda$ is contributed to the sum. When $S\neq\emptyset$, the graph $(V(H_\lambda),S)$ has fewer than $\ell(\lambda)$ connected components (since there are no loops in $S$), and so a lower-order term is contributed to the expansion.

Hence we can write each $X_{(H_\lambda,\lambda)}$ as the sum of $p_\lambda$ and possible lower-order terms. Thus for each $n$, the matrix expressing the extended chromatic symmetric functions $\{X_{(H_\lambda,\lambda)}\}_{\lambda\vdash n}$ in terms of the power sum symmetric functions $\{p_\lambda\}_{\lambda\vdash n}$ is lower triangular with $1$'s on the diagonal.

Therefore $\{X_{(H_\lambda,\lambda)}\}_{\lambda\vdash n}$ is a basis of $\Sym^n$ for each $n$, and so $\{X_{(H_\lambda,\lambda)}\}_{\lambda\vdash n\ge 0}$ is a basis for $\Sym$.
\end{proof}

\begin{example}\label{ex:matrix}
Corollary~\ref{cor:wpbas} is the case where each $H_\lambda$ is the labelled path $P_{\ell(\lambda)}$ on $\ell(\lambda)$ vertices. Using Proposition~\ref{prop:pwr}, we can compute the entries of the matrix describing the family $\{X_{(P_{\ell(\lambda)},\lambda)}\}_{\lambda\vdash 4}$ in terms of the basis of power sum symmetric functions $\{p_\lambda\}_{\lambda\vdash4}$ of $\Sym^4$.
\begin{equation*}
    \begin{pmatrix}
    X_{(P_1,(4))}\\
    X_{(P_2,(3,1))}\\
    X_{(P_2,(2,2))}\\
    X_{(P_3,(2,1,1))}\\
    X_{(P_4,(1,1,1,1))}\\
    \end{pmatrix}
    =
    \begin{pmatrix}
    1\\
    -1&1 &  & &\\
    -1& &1\\
    1&-1 & -1&1 &\\
    -1&2 & 1&-3 &1\\
    \end{pmatrix}
    \begin{pmatrix}
    p_4\\
    p_{(3,1)}\\
    p_{(2,2)}\\
    p_{(2,1,1)}\\
    p_{(1,1,1,1)}\\
    \end{pmatrix}
\end{equation*}
Note that indices are ordered by the number of parts in a partition, and so our matrix is indeed lower triangular with $1$'s on the diagonal.
\end{example}

\begin{remark}\label{rem:matrix}
The above example shows how Corollary~\ref{cor:wpbas} could have alternatively been proved by an application of Theorem~\ref{the:wbas}.

One proof of the fact that the \finrev{ribbon} Schur functions indexed by partitions form a basis uses a lower triangularity argument with   \eqref{eq:rintoh}, which expands a ribbon Schur function into the basis of complete homogeneous symmetric functions.

As we shall soon see in Remark~\ref{rem:moreU}, the automorphism $U$, which takes each $r_\alpha\mapsto X_{(P_{\ell(\alpha)},\alpha)}$, also takes each $h_\lambda\mapsto p_\lambda$. Under this automorphism, showing \rev{the result of the previous paragraph} is equivalent to showing that the extended chromatic symmetric functions of weighted paths indexed by partitions form a basis via the proof of Theorem~\ref{the:wbas}.
\end{remark}

\section{Neat changes of basis}\label{sec:neatbases}
To work with the chromatic bases of Theorem~\ref{the:chr}, it is important to understand how the classical bases of $\Sym$ expand in the new bases being considered. Crew and Spirkl noted in their proof of \cite[Lemma 3]{CS} that the classical power sum symmetric functions $p_\lambda$ are exactly the extended chromatic symmetric functions $X_{(N_{\ell(\lambda)},\lambda)}$. In particular, the $i$th power sum symmetric function $p_i$ is exactly the extended chromatic symmetric function of a single vertex of weight $i$.

To write the power sum symmetric function $p_i$ in terms of the chromatic symmetric functions of unweighted graphs, we might think to apply Theorem~\ref{the:incexc} on the single vertex of weight $i$, using an expansion onto $i$ independent vertices of equal weight $1$. We would also need an edge set $E'$ connecting the $i$ vertices.

To obtain a tidy formula expressing $p_i$ in the basis $\{X_{G_\lambda}\}_{\lambda\vdash n\ge 0}$ for some nifty family $\{G_n\}_{n\ge 1}$ of unweighted graphs, considering the formula in Theorem~\ref{the:incexc}, it would be desirable if \textit{every} graph on $i$ vertices with edges from a subset of $E'$ were a graph $G_\lambda$ obtained from our family for some partition $\lambda$. This motivates the following definition, which we give in a slightly more general form to allow for the weighted case.

\begin{definition}\label{def:neat}
A nifty family $\{(G_n,w_n)\}_{n\ge1}$ is \textit{neat} if for all $n\ge 1$, for all subsets $S\subseteq E(G_n)$, we have $(G_n-S,w_n)$ is isomorphic to $(G_\lambda,w_\lambda)$ for some partition $\lambda\vdash n$.
\end{definition}

\begin{proposition}\label{prop:neat}
The only neat families of unweighted graphs are the family of paths $\{P_n\}_{n\ge1}$ and the family of stars $\{S_n\}_{n\ge1}$.
\end{proposition}

\begin{proof}
First observe that the family of paths $\{P_n\}_{n\ge1}$ and the family of stars $\{S_n\}_{n\ge1}$ are neat.
Now note that any neat family $\{G_n\}_{n\ge 1}$ of unweighted graphs must consist entirely of trees. To see this, consider any $G_n$ and let $\epsilon$ be any edge of $G_n$. By Definition~\ref{def:neat}, $G_n-\epsilon$ must be isomorphic to $G_\lambda$ for some $\lambda\vdash n$. The graphs $G_n-\epsilon$ and $G_n$ cannot be isomorphic, since they have different numbers of edges. Hence $G_n-\epsilon$ must be isomorphic to $G_\lambda$ for some   partition $\lambda$ satisfying $\ell(\lambda)>1$. In particular, $G_\lambda$ has more than one connected component, and so must be disconnected. Since $G_n$ is a connected graph such that the deletion of any edge $\epsilon$ disconnects it, $G_n$ must be a tree, by definition.

The only trees on $1$, $2$, and $3$ vertices, respectively are $P_1=S_1$, $P_2=S_2$, and $P_3=S_3$, up to isomorphism. The two nonisomorphic trees on $4$ vertices are $P_4$ and $S_4$.

Let $n\geq 4$, for ease of notation denote $G_n$ by $T$, and let
$$m= \min \{ \maxdeg(T), \width(T) \}$$ where $\maxdeg$ is the maximum degree and $\width$ is the length of the longest path in $T$. If $m\geq 3$ then consider the subgraph \rev{$T'$ of $T$} induced by the vertex of maximum degree and 3 of its neighbours, that is $T'= S_4$. Now consider the subgraph of \rev{$T''$ of $T$} induced by the first $4$ vertices on a path of longest length, that is $T'' = P_4$. If we have  a neat family of graphs then by   definition  it follows that $$S_4 =T'= G_4 = T''= P_4$$\SvW{giving us a contradiction.}

Hence $m\leq 2$, which implies that either $\maxdeg(T) \leq 2$, in which case $G_n$    is a path, or $\width(T) \leq 2$ in which case $G_n$    is a star. Considering our family is neat, it follows by definition that if $G_n = P_n$ then $G_i = P_i$ for all $i<n$, and if $G_n = S_n$ then $G_i = S_i$ for all $i<n$. Finally, let $n$ tend to infinity.

Thus any neat family of unweighted graphs $\{G_n\}_{n\ge1}$ must either be the family of paths $\{P_n\}_{n\ge1}$ or the family of stars $\{S_n\}_{n\ge1}$. 
\end{proof}

We now proceed with our plan to expand the power sum symmetric functions in terms of the chromatic bases generated by neat families of unweighted graphs via Theorem~\ref{the:incexc}. However, we first note that Theorem~\ref{the:incexc} is rather similar in form to the formula in Proposition~\ref{prop:pwr}. This similarity gives us something neat:

\begin{theorem}[Chromatic reciprocity] \label{the:rec}
Let $\{(G_n,w_n)\}_{n\ge1}$ be a neat family. Then the unique linear transformation $\varphi:\Sym\to\Sym$ mapping $p_\lambda\mapsto X_{(G_\lambda,w_\lambda)}$ for each partition $\lambda$ is exactly the unique linear transformation mapping $X_{(G_\lambda,w_\lambda)}\mapsto p_\lambda$ for each partition $\lambda$. In particular, $\varphi$ is an involutory automorphism of $\Sym$ as a graded algebra.
\end{theorem}

\begin{proof}
By Proposition~\ref{prop:pwr}, for each $n\ge 1$,
\begin{equation*}
    X_{(G_n,w_n)} = \sum_{S\subseteq E(G_n)} (-1)^{|S|}p_{\lambda((V(G_n),S),w_n)}.
\end{equation*}

Now consider the single vertex of weight $n$, which has extended chromatic symmetric function $p_n$. The weighted graph $((V(G_n),\emptyset),w_n)$ with no edges is an expansion of the single weighted vertex, and the edge set $E(G_n)$ connects the vertices of the expansion. By Theorem~\ref{the:incexc}, we obtain
\begin{equation*}
    p_n=\sum_{S\subseteq E(G_n)} (-1)^{|S|} X_{((V(G_n),S),w_n)}.
\end{equation*}

Since $\{(G_n,w_n)\}_{n\ge 1}$ is a neat family, for each subset $S\subseteq E(G_n)$, we must have that $((V(G_n),S),w_n)$ is the weighted graph $(G_\lambda,w_\lambda)$ for some   partition $\lambda\vdash n$. Since each $(G_\lambda,w_\lambda)$ is the disjoint union of connected weighted graphs of total weight $\lambda_i$ for each part of $\lambda$, the only possibility for this partition is $\lambda((V(G_n),S),w_n)$, whose parts are the sums of the vertex weights of each connected component of $((V(G_n),S),w_n)$.

Hence if 
\begin{equation*}
    X_{(G_n,w_n)}=\sum_{\lambda\vdash n} c_\lambda p_\lambda
\end{equation*}
we must identically have 
\begin{equation*}
    p_n = \sum_{\lambda\vdash n} c_\lambda X_{(G_\lambda,w_\lambda)}
\end{equation*}
with the same coefficients $c_\lambda$.

Let $\varphi:\Sym\to\Sym$ be the unique linear map taking each $p_\lambda \mapsto X_{(G_\lambda,w_\lambda)}$. Since $\{p_\lambda\}_{\lambda\vdash n\ge 0}$ and $\{X_{(G_\lambda,w_\lambda)}\}_{\lambda\vdash n\ge 0}$ are multiplicative bases of $\Sym$ by definition and Theorem~\ref{the:wbase} respectively, the map $\varphi$ is an automorphism of $\Sym$ as a graded algebra.

For each $n\ge 1$, if we apply $\varphi$ to both sides of our expansion of $X_{(G_n,w_n)}$ into the power sum symmetric functions, we find that
\begin{equation*}
    \varphi(X_{(G_n,w_n)})=\sum_{\lambda\vdash n}c_\lambda X_{(G_\lambda,w_\lambda)}=p_n.
\end{equation*}

Since $\varphi$ is an automorphism of $\Sym$ as a graded algebra, it respects multiplication, and therefore takes each $X_{(G_\lambda,w_\lambda)}\mapsto p_\lambda$. That is, $\varphi$ is an involution on $\Sym$.
\end{proof}

We can now use chromatic reciprocity to deduce change of basis formulae. Note that   \eqref{eq:pathtop} is also implicit in the work of  Chmutov, Duzhin and Lando \cite{Chmutovetal}.

\begin{proposition}\label{prop:ppwr}
The power sum symmetric functions expand into the basis generated by the neat family of paths $\{P_n\}_{n\ge1}$ via
\begin{equation*} 
    p_\lambda = \sum_{\alpha\preccurlyeq\lambda} (-1)^{|\lambda| - \ell(\alpha)} X_{P_{\widetilde\alpha}}
\end{equation*}
and similarly
\begin{equation*}\label{eq:pathp}
    X_{P_\lambda} = \sum_{\alpha\preccurlyeq\lambda} (-1)^{|\lambda| - \ell(\alpha)} p_{\widetilde\alpha}.
\end{equation*}\end{proposition}

\begin{proof}
Recall that if $G$ is a graph with $n$ vertices then $X_{(G, (1^n))}=X_G$. Hence by Proposition~\ref{prop:pwr} we have that
\begin{equation}\label{eq:pathtop}
X_{P_n}= X_{(P_n, (1^n))}= \sum _{S\subseteq E(G)} (-1)^{|S|} p _{\lambda((V(P_n),S), (1^n))} = \sum _{\beta \vDash n} (-1) ^{n-\ell(\beta)} p _{\widetilde\beta}.
\end{equation}

We can compute the expansion of $X_{P_\lambda}$ into the power sum symmetric functions by multiplying the expansions of each $X_{P_{\lambda _i}}$. We obtain
\begin{align*}
    X_{P_\lambda}&= X_{P_{\lambda_1}}\cdots X_{P_{\lambda_{\ell(\lambda)}}}\\
    &=\prod_{i=1}^{\ell(\lambda)}\left(\sum _{\beta \vDash \lambda_i} (-1) ^{\lambda_i-\ell(\beta)} p _{\widetilde\beta}\right)\\
    &=\sum_{\alpha\preccurlyeq\lambda} (-1)^{|\lambda| - \ell(\alpha)} p_{\widetilde\alpha}.
\end{align*}

Since the family of paths is neat, by Theorem~\ref{the:rec} we also have  
\begin{equation*}
   p_\lambda = \sum_{\alpha\preccurlyeq\lambda} (-1)^{|\lambda| - \ell(\alpha)} X_{P_{\widetilde\alpha}}.\end{equation*}
\end{proof}

\begin{remark}\label{rem:moreU}
The  \finrev{formula} expressing the complete homogeneous symmetric function $h_\lambda$ in terms of the elementary symmetric functions is given by \rev{\cite[Definition 3.2.6]{LMvW}}
\begin{equation*}
    h_\lambda = \sum_{\alpha\preccurlyeq\lambda} (-1)^{|\lambda| - \ell(\alpha)} e_{\widetilde\alpha}.
\end{equation*}

Thus the linear map $U$, introduced in \eqref{eq:Umap}, which takes each $e_\lambda\mapsto X_{P_\lambda}$, also takes $h_\lambda\mapsto p_\lambda$ for each partition $\lambda$ by Proposition~\ref{prop:ppwr}. Applying $U$ to   \eqref{eq:rintoh}, and recalling   \eqref{eq:rtoXP}, we obtain
\begin{equation*}
    X_{(P_{\ell(\alpha)},\alpha)} = \sum_{\beta \succcurlyeq \alpha} (-1)^{\ell(\alpha)-\ell(\beta)} p_{\widetilde \beta}.
\end{equation*}

It is perhaps easier, but  less instructive, to deduce the relationship between the ribbon Schur functions and the extended chromatic symmetric functions of weighted functions from the above formula, which could have been found via an application of Proposition~\ref{prop:pwr}  {in a way similar to our proof of Proposition~\ref{prop:ppwr}}.

Also note that we can deduce the linear involution taking each $X_{P_\lambda}\mapsto p_\lambda$. The involution must be the map $U\omega U^{-1}$, where $\omega$ is the linear involution on $\Sym$ taking each $e_\lambda\mapsto h_\lambda$, since   $$p_\lambda = U(h_\lambda)=U\omega(e_\lambda)=U\omega U^{-1}(X_{P_\lambda}).$$
\end{remark}

In \cite[Theorem 8]{ChovW} of Cho and van Willigenburg's original paper introducing chromatic bases, they computed an expansion of the path basis into the power sum symmetric functions, which appears in a different form from Proposition~\ref{prop:ppwr}. Because the power sum symmetric functions are linearly independent, the expansions must ultimately be the same.

Another expansion they computed was the expansion of the star basis into the power sum symmetric functions. Specifically, they found in \cite[Theorem 8]{ChovW} for $n+1\geq 1$ that
\begin{equation*}
    X_{S_{n+1}} = \sum_{r=0}^{n} (-1)^{r} {n\choose r} p_{(r+1,1^{n-r})}.
\end{equation*}

For our purposes, we will consider the equivalent form for $n\ge 1$,
\begin{equation*}
    X_{S_n} = \sum_{r=1}^{n} (-1)^{r-1} {n-1\choose r-1} p_{(r,1^{n-r})}.
\end{equation*}

We can compute an expansion for $X_{S_\lambda}$ by multiplying expansions of the above form for each part of $\lambda$. Since the stars form a neat family, this also gives an expansion of the power sum symmetric functions in terms of the star basis by Theorem~\ref{the:rec} as follows.

\begin{proposition}\label{prop:spwr}
The power sum symmetric functions expand into the basis generated by the neat family of stars $\{S_n\}_{n\ge1}$ via 
\begin{equation*}
    p_\lambda = \sum_{\alpha \subseteq \lambda} (-1)^{|\alpha|-\ell(\lambda)} {\lambda_1-1\choose\alpha_1-1}\cdots{\lambda_{\ell(\lambda)}-1\choose\alpha_{\ell(\lambda)}-1} X_{S_{\widetilde{\alpha}\cdot(1^{|\lambda|-|\alpha|})}}
\end{equation*}
and similarly
\begin{equation*}
    X_{S_\lambda} = \sum_{\alpha \subseteq \lambda} (-1)^{|\alpha|-\ell(\lambda)} {\lambda_1-1\choose\alpha_1-1}\cdots{\lambda_{\ell(\lambda)}-1\choose\alpha_{\ell(\lambda)}-1} p_{\widetilde{\alpha}\cdot(1^{|\lambda|-|\alpha|})}
\end{equation*}
\rev{where $\alpha \subseteq \lambda$ means $\ell(\alpha)=\ell(\lambda)$ and $\alpha_1\le \lambda_1,\dots,\alpha_{\ell(\lambda)}\le \lambda_{\ell(\lambda)}$.}
\end{proposition}
\begin{proof}
We can compute the expansion of $X_{S_\lambda}$ into the power sum symmetric functions by multiplying the expansions of each $X_{S_{\lambda_i}}$. We obtain
\begin{align*}
    X_{S_\lambda}&= X_{S_{\lambda_1}}\cdots X_{S_{\lambda_{\ell(\lambda)}}}\\
    &=\prod_{i=1}^{\ell(\lambda)}\left(\sum_{\alpha_i=1}^{\lambda_i} (-1)^{\alpha_i-1} {\lambda_i-1\choose \alpha_i-1} p_{(\alpha_{i},1^{\lambda_{i}-\alpha_{i}})}\right)\\
    &=\sum_{\alpha\subseteq\lambda}\left(\prod_{i=1}^{\ell(\lambda)}(-1)^{\alpha_i-1} {\lambda_i-1\choose \alpha_i-1}p_{(\alpha_{i},1^{\lambda_{i}-\alpha_{i}})}\right)
\end{align*}
since summing over all tuples $(\alpha_1,\dots,\alpha_{\ell(\lambda)})$ of positive integers satisfying $\alpha_1\le \lambda_1,\dots,\alpha_{\ell(\lambda)}\le \lambda_{\ell(\lambda)}$ is the same as summing over all compositions $\alpha$ contained in $\lambda$.

Expanding out the product in each term of the summation gives us
\begin{equation*}
    X_{S_\lambda} = \sum_{\alpha \subseteq \lambda} (-1)^{|\alpha|-\ell(\lambda)} {\lambda_1-1\choose\alpha_1-1}\cdots{\lambda_{\ell(\lambda)}-1\choose\alpha_{\ell(\lambda)}-1} p_{\widetilde{\alpha}\cdot(1^{|\lambda|-|\alpha|})}.
\end{equation*}
Since the  family of stars is neat, by Theorem~\ref{the:rec} we also have  
\begin{equation*}
    p_\lambda = \sum_{\alpha \subseteq \lambda} (-1)^{|\alpha|-\ell(\lambda)} {\lambda_1-1\choose\alpha_1-1}\cdots{\lambda_{\ell(\lambda)}-1\choose\alpha_{\ell(\lambda)}-1} X_{S_{\widetilde{\alpha}\cdot(1^{|\lambda|-|\alpha|})}}.
\end{equation*}
\end{proof}

\section{Composition of graphs and equality of weighted graphs}\label{sec:compgra}
In \cite{CS} Crew and Spirkl stated that they did not know of two weighted trees with the same extended chromatic symmetric function that were nonisomorphic. \SvW{In \cite{ACSZ} they, with Aliste-Prieto and Zamora, found such a pair.} \rev{Independently, the authors of this article found a different pair, below. Note that the two graphs are nonisomorphic because in the one on the left the two vertices of degree 3 are not adjacent, whereas they are in the one on the right.}

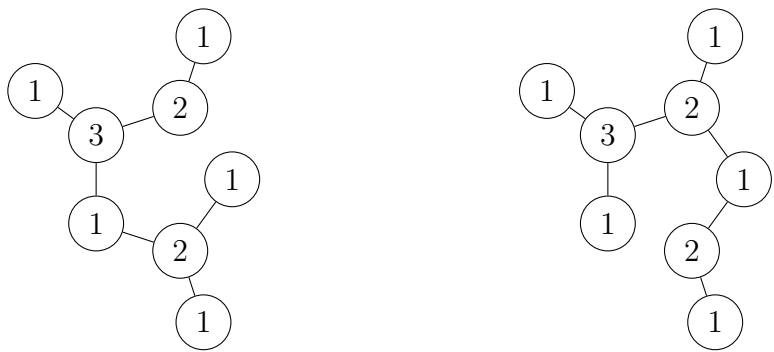
\begin{center}
\begin{tikzpicture}
\def \n {5}
\def \radius {1cm}
\def\ray{{1,2,3,1,2}}

\foreach \s in {1,...,\n}
{
  \node[draw, circle] (\s) at ({360/\n * (\s - 1)}:\radius) {\pgfmathparse{\ray[\s-1]} \pgfmathresult};
}
\node[draw, circle] (a) at ({360/\n * (1}:2cm) {1};
\node[draw, circle] (b) at ({360/\n * (2}:2cm) {1};
\node[draw, circle] (c) at ({360/\n * (4}:2cm) {1};
\path [-] (2) edge (3);
\path [-] (2) edge (a);
\path [-] (3) edge (4);
\path [-] (3) edge (b);
\path [-] (4) edge (5);
\path [-] (5) edge (1);
\path [-] (5) edge (c);
\end{tikzpicture}
\quad\quad\quad\quad\quad\quad\quad\quad
\begin{tikzpicture}
\def \n {5}
\def \radius {1cm}
\def\ray{{1,2,3,1,2}}

\foreach \s in {1,...,\n}
{
  \node[draw, circle] (\s) at ({360/\n * (\s - 1)}:\radius) {\pgfmathparse{\ray[\s-1]} \pgfmathresult};
}
\node[draw, circle] (a) at ({360/\n * (1}:2cm) {1};
\node[draw, circle] (b) at ({360/\n * (2}:2cm) {1};
\node[draw, circle] (c) at ({360/\n * (4}:2cm) {1};
\path [-] (1) edge (2);
\path [-] (2) edge (3);
\path [-] (2) edge (a);
\path [-] (3) edge (4);
\path [-] (3) edge (b);
\path [-] (5) edge (1);
\path [-] (5) edge (c);
\end{tikzpicture}
\end{center}

\rev{As we will see in Theorem~\ref{the:eq}, the} equality of extended chromatic symmetric functions of the two weighted trees described above can be deduced from a more general construction of families of weighted graphs with equal extended chromatic symmetric functions, and is a generalization of the binary operation $\circ$.

\begin{definition}\label{def:composeGraph}
Let $(G,w)$ be a weighted graph with distinguished (not necessarily distinct) vertices $a$ and $z$. Given a nonempty   composition $\alpha$, we define the weighted graph $\alpha \circ (G,w)$ as follows.

Consider the disjoint union of $|\alpha|$ copies of $(G,w)$, and let the copies of $a$ and $z$ in the $i$th copy of $(G,w)$ be labelled $a_i$ and $z_i$. Add an edge $z_ia_{i+1}$ for each $i\in [|\alpha|-1]$. Then $\alpha\circ (G,w)$ denotes the weighted graph resulting from contracting the edges $z_ia_{i+1}$ for all $i\in\text{set}(\alpha^c)$.
\end{definition}

\begin{example}\label{ex:comp}
Let $(G,w)$ be the weighted path $(P_3,(1,2,1))$. Choose $a$ to be either vertex of weight $1$ and $z$ to be the vertex of weight $2$. Let $\alpha$ be the composition $(1,2)$. Below is the result of adding in the edges $z_ia_{i+1}$ to the disjoint union of $3$ copies of $(G,w)$.
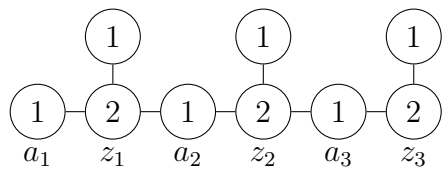
\begin{center}
\begin{tikzpicture}
\node[shape=circle,draw=black] (A) at (1,0) {1};
\node[shape=circle,draw=black] (B) at (2,0) {2};
\node[shape=circle,draw=black] (C) at (3,0) {1};
\node[shape=circle,draw=black] (D) at (4,0) {2};
\node[shape=circle,draw=black] (E) at (5,0) {1};
\node[shape=circle,draw=black] (F) at (6,0) {2};
\node[shape=circle,draw=black] (G) at (2,1) {1};
\node[shape=circle,draw=black] (H) at (4,1) {1};
\node[shape=circle,draw=black] (I) at (6,1) {1};
\path [-] (A) edge (B);
\path [-] (B) edge (C);
\path [-] (B) edge (G);
\path [-] (D) edge (C);
\path [-] (D) edge (E);
\path [-] (D) edge (H);
\path [-] (F) edge (E);
\path [-] (F) edge (I);
\node [] at (1,-.6) {$a_1$};
\node [] at (2,-.6) {$z_1$};
\node [] at (3,-.6) {$a_2$};
\node [] at (4,-.6) {$z_2$};
\node [] at (5,-.6) {$a_3$};
\node [] at (6,-.6) {$z_3$};
\end{tikzpicture}
\end{center}

Since $\text{set}((1,2)^c)=\{2\}$, we obtain the weighted graph $(1,2)\circ(G,w)$, drawn below, by contracting the edge $z_2a_3$.
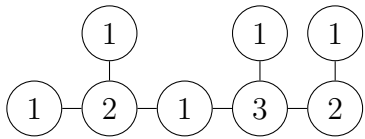
\begin{center}
\begin{tikzpicture}
\node[shape=circle,draw=black] (B) at (2,0) {1};
\node[shape=circle,draw=black] (C) at (3,0) {2};
\node[shape=circle,draw=black] (D) at (4,0) {1};
\node[shape=circle,draw=black] (E) at (5,0) {3};
\node[shape=circle,draw=black] (F) at (6,0) {2};
\node[shape=circle,draw=black] (G) at (3,1) {1};
\node[shape=circle,draw=black] (H) at (5,1) {1};
\node[shape=circle,draw=black] (I) at (6,1) {1};
\path [-] (B) edge (C);
\path [-] (C) edge (G);
\path [-] (D) edge (C);
\path [-] (D) edge (E);
\path [-] (E) edge (H);
\path [-] (F) edge (E);
\path [-] (F) edge (I);
\end{tikzpicture}
\end{center}
\end{example}

Note that given any two nonempty compositions $\alpha$ and $\beta$, the elements of $\text{set}((\alpha\odot\beta)^c)$ consist of the elements of $\text{set}((\alpha\cdot\beta)^c)$, in addition to the element $|\alpha|$. Thus by Definition~\ref{def:composeGraph}, the weighted graph $(\alpha\odot\beta)\circ(G,w)$ can be obtained from $(\alpha\cdot\beta)\circ(G,w)$ by contracting the edge $z_{|\alpha|}a_{|\alpha|+1}$. \finrev{If} we instead delete the edge $z_{|\alpha|}a_{|\alpha|+1}$ from $(\alpha\cdot\beta)\circ(G,w)$, we obtain the disjoint union $\alpha\circ(G,w)\cup\beta\circ(G,w)$.

By the deletion-contraction rule of Proposition~\ref{prop:dc}, we then have
\begin{equation*}
    X_{(\alpha\cdot\beta)\circ(G,w)} = X_{\alpha\circ(G,w)\cup\beta\circ(G,w)}- X_{(\alpha\odot\beta)\circ(G,w)} = X_{\alpha\circ(G,w)}X_{\beta\circ(G,w)}- X_{(\alpha\odot\beta)\circ(G,w)}.
\end{equation*}
The similarity between this equation and the product rule for   ribbon Schur functions will allow us to deduce the following theorem. \SvW{The first part of this theorem has also been discovered independently, but this time in the language of $\mathcal{L}$-polynomials \cite{ACSZ}.}

\begin{theorem}\label{the:eq}
Let $(G,w)$ be a weighted graph with distinguished vertices $a$ and $z$. Then for any two nonempty compositions $\alpha$ and $\beta$,
$$X_{\alpha\circ(G,w)} = X_{\beta\circ(G,w)}
\hbox{\rm \quad if \quad}
 \alpha\sim\beta.$$
If moreover the underlying graph $G$ of $(G,w)$ is simple and connected, then this strengthens to $$X_{\alpha\circ(G,w)} = X_{\beta\circ(G,w)}
\hbox{\rm \quad if and only if \quad}
 \alpha\sim\beta.$$
\end{theorem}

\begin{proof}
We will show that $X_{\alpha\circ(G,w)}$ is the image of the ribbon Schur function $r_\alpha$ under the algebra endomorphism $U_{(G,w)}:\Sym\to\Sym$ mapping $h_i\mapsto X_{(i)\circ(G,w)}$ for each $i\ge 1$.

We proceed by induction on the length $\ell$ of $\alpha$. If $\ell=1$, then $\alpha=(\alpha_1)$ consists of a single part. Then indeed $X_{\alpha\circ(G,w)}=U_{(G,w)}(h_{\alpha_1})=U_{(G,w)}(r_{\alpha})$, so the base case holds.

Now suppose $\alpha$ is of length $\ell\ge 2$ and the inductive hypothesis holds for all compositions of length $<\ell$. Then $$X_{\alpha\circ(G,w)}=X_{(\alpha_1)\circ(G,w)}X_{(\alpha_2,\dots,\alpha_{\ell})\circ(G,w)}-X_{(\alpha_1+\alpha_2,\dots,\alpha_{\ell})\circ(G,w)}$$ by the deletion-contraction rule, Proposition~\ref{prop:dc}. Applying the inductive hypothesis, we obtain
\begin{align*}
    X_{\alpha\circ(G,w)}&=U_{(G,w)}(r_{(\alpha_1)})U_{(G,w)}(r_{(\alpha_2,\dots,\alpha_{\ell})})-U_{(G,w)}(r_{(\alpha_1+\alpha_2,\dots,\alpha_{\ell})})\\
    &=U_{(G,w)}(r_{(\alpha_1)}r_{(\alpha_2,\dots,\alpha_{\ell})}-r_{(\alpha_1+\alpha_2,\dots,\alpha_{\ell})})\\
    &= U_{(G,w)}(r_\alpha)
\end{align*}
since   ribbon Schur functions satisfy the product rule described in   \eqref{eq:ribMult}.

Hence if $\alpha\sim\beta$, then $r_\alpha = r_{\beta}$ and so $$X_{\alpha\circ(G,w)}=U_{(G,w)}(r_\alpha)=U_{(G,w)}(r_\beta)=X_{\beta\circ(G,w)}.$$

When $G$ is simple and connected, the weighted graphs $\{(i)\circ(G,w)\}_{i\ge1}$ are each  {simple and} connected with distinct total weights. \rev{The weighted graphs $\{(i)\circ(G,w)\}_{i\ge1}$ can be seen as a subset of a nifty family, thus by Theorem~\ref{the:wbase} the extended chromatic symmetric functions $\{X_{(i)\circ (G,w)}\}_{i\ge1}$ are algebraically independent.} In that case, $U_{(G,w)}$ is injective and so we also have $X_{\alpha\circ(G,w)}=X_{\beta\circ(G,w)}$ only if $\alpha\sim\beta$.
\end{proof}

\begin{example}\label{ex:eq}
Let $(G,w)$ be the same weighted graph as in Example~\ref{ex:comp}, with the same choice of vertices $a$ and $z$. Since $(1,2)\sim (2,1)$, by Theorem~\ref{the:eq} the weighted graphs $(1,2)\circ (G,w)$ and $(2,1)\circ(G,w)$ below have equal extended chromatic symmetric functions.
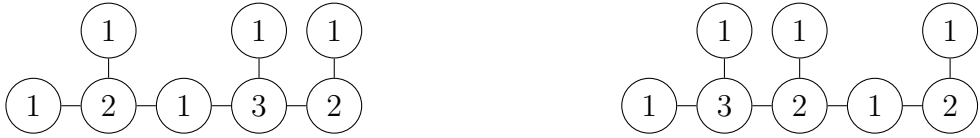
\begin{center}
\begin{tikzpicture}
\node[shape=circle,draw=black] (B) at (2,0) {1};
\node[shape=circle,draw=black] (C) at (3,0) {2};
\node[shape=circle,draw=black] (D) at (4,0) {1};
\node[shape=circle,draw=black] (E) at (5,0) {3};
\node[shape=circle,draw=black] (F) at (6,0) {2};
\node[shape=circle,draw=black] (G) at (3,1) {1};
\node[shape=circle,draw=black] (H) at (5,1) {1};
\node[shape=circle,draw=black] (I) at (6,1) {1};
\path [-] (B) edge (C);
\path [-] (C) edge (G);
\path [-] (D) edge (C);
\path [-] (D) edge (E);
\path [-] (E) edge (H);
\path [-] (F) edge (E);
\path [-] (F) edge (I);
\end{tikzpicture}
\quad\quad\quad\quad\quad\quad\quad\quad
\begin{tikzpicture}
\node[shape=circle,draw=black] (B) at (2,0) {1};
\node[shape=circle,draw=black] (C) at (3,0) {3};
\node[shape=circle,draw=black] (D) at (4,0) {2};
\node[shape=circle,draw=black] (E) at (5,0) {1};
\node[shape=circle,draw=black] (F) at (6,0) {2};
\node[shape=circle,draw=black] (G) at (3,1) {1};
\node[shape=circle,draw=black] (H) at (4,1) {1};
\node[shape=circle,draw=black] (I) at (6,1) {1};
\path [-] (B) edge (C);
\path [-] (C) edge (G);
\path [-] (D) edge (C);
\path [-] (D) edge (E);
\path [-] (D) edge (H);
\path [-] (F) edge (E);
\path [-] (F) edge (I);

\end{tikzpicture}
\end{center}
\finrev{This is our example from the start of the section.}
\rev{The example of Aliste-Prieto, Crew, Spirkl and Zamora in \cite{ACSZ} looks very similar, however in their example $(G,w) = (P_3,(1,1,2))$ rather than $(G,w) = (P_3,(1,2,1))$.}
\end{example}

\begin{remark}\label{rem:eq}
Because each $e_i$ is equal to the ribbon Schur function $r_{(1^i)}$, we also see that $U_{(G,w)}$ is the algebra endomorphism taking $e_i\mapsto X_{(1^i)\circ(G,w)}$ for each $i\ge1$. 

When $(G,w)$ is the graph of a single vertex of weight $1$, the map $U_{(G,w)}$ is exactly the automorphism $U:\Sym\to\Sym$ introduced in \eqref{eq:Umap}, which takes each $h_i\mapsto p_i$ and each $e_i\mapsto X_{P_i}$ for $i\ge1$. Applying Theorem~\ref{the:eq} to this case gives Theorem~\ref{the:wpclass}, the classification of equality of extended chromatic symmetric functions of weighted paths.
\end{remark}

As a final note,  we have seen how known linear relations between ribbon Schur functions have given us a plethora of results for weighted paths. This raises the natural question: What linear relations between ribbon Schur functions can we obtain from weighted paths? For example, by applying the deletion-contraction rule of Propositions~\ref{prop:dc} on two different edges of the cycle on 3 vertices from Example~\ref{ex:dc} with vertex weights $3,2,1$, we have $$X_{(P_3,(2,1,3))}-X_{(P_2,(1,5))}=X_{(P_3,(2,3,1))}-X_{(P_2,(3,3))}$$where the edge we do not apply to is the one between the vertices of weights 1 and 3, and so $$r_{(2,1,3)}+r_{(3,3)}=r_{(2,3,1)}+r_{(1,5)}.$$

\section{Acknowledgments}\label{sec:ack}  \SvW{The authors would like to thank Jos\'{e} Aliste-Prieto, Logan Crew, Sophie Spirkl and Jos\'{e} Zamora for helpful conversations.} \rev{They would also like to thank the referee for the care they took with our paper and their very thoughtful comments.}

\section*{Funding}\label{cash}
All authors were supported in part by the National Sciences and Engineering Research Council of \rev{Canada.}

\section*{References}\label{refs}
\bibliographystyle{plain}

\def\cprime{$'$}

\end{document}